\documentclass[twoside,11pt]{article}

\usepackage{amsmath,amssymb,amsthm,amsfonts}
\usepackage{appendix}

\newtheorem{lemma}{Lemma}[section]
\newtheorem{theorem}{Theorem}[section]

\newtheorem{proposition}{Proposition}[section]
\newtheorem{remark}{Remark}[section]
\newtheorem{corollary}{Corollary}[section]

\numberwithin{equation}{section}

\newcommand{\dis}{\displaystyle}

\newcommand{\rmi}{{\rm i}}
\newcommand{\rmre}{{\rm Re}}

\newcommand{\R}{\mathbb{R}}

\newcommand{\Z}{\mathbb{Z}}

\newcommand{\T}{\mathbb{T}}

\newcommand{\CA}{\mathcal{A}}
\newcommand{\CD}{\mathcal{D}}
\newcommand{\CE}{\mathcal{E}}

\newcommand{\CH}{\mathcal{H}}

\newcommand{\CM}{\mathcal{M}}

\newcommand{\SB}{\text{\rm \textsf{B}}}
\newcommand{\SX}{\text{\rm \textsf{X}}}
\newcommand{\SY}{\text{\rm \textsf{Y}}}
\newcommand{\ST}{\text{\rm \textsf{T}}}
\newcommand{\SL}{\text{\rm \textsf{L}}}
\newcommand{\SI}{\text{\rm \textsf{I}}}
\newcommand{\SP}{\text{\rm \textsf{P}}}
\newcommand{\SQ}{\text{\rm \textsf{Q}}}

\newcommand{\na}{\nabla}

\newcommand{\al}{\alpha}

\newcommand{\om}{\omega}
\newcommand{\la}{\lambda}
\newcommand{\de}{\delta}
\newcommand{\si}{\sigma}
\newcommand{\pa}{\partial}

\newcommand{\eps}{\epsilon}

\newcommand{\De}{\Delta}
\newcommand{\Ga}{\Gamma}

\begin{document}

\title{\bf Hypocoercivity of Linear Degenerately Dissipative Kinetic Equations}
\author{\textsc{Renjun Duan} \\[1mm]
Johann Radon Institute for Computational and Applied Mathematics\\
Austrian Academy of Sciences\\
Altenbergerstrasse 69, A-4040 Linz, Austria\\
{\it E-mail: mathrjduan@hotmail.com}}
\date{December 3, 2009}
\maketitle

\begin{abstract}
In this paper, we study the hypocoercivity for a class of linear
kinetic equations with both transport and degenerately dissipative
terms. As concrete examples, the relaxation operator, Fokker-Planck
operator and linearized Boltzmann operator are considered. By
constructing equivalent temporal energy functionals, time rates of
the solution approaching equilibrium in some Hilbert spaces are
obtained when the spatial domain takes the whole space or torus and
when there is a confining force or not. The main tool of the proof
is the macro-micro decomposition combined with Kawashima's argument
on dissipation of the hyperbolic-parabolic system. Finally, a
Korn-type inequality with probability measure is provided to deal
with dissipation of momentum components.
\end{abstract}

{\small {\it Keywords:} {Hypocoercivity; semigroup; kinetic
equation; dissipation}

%\medskip

%AMS Subject Classification (2000): ***}

\tableofcontents

%\newpage

\section{Introduction}

\noindent{\it 1.1 Problems}. In this paper, we consider the linear
kinetic equation with both transport and dissipative terms in the
form of
\begin{equation}\label{s1.eq}
    \pa_t u+ \ST u= \SL u.
\end{equation}
Here, the unknown is $u=u(t,x,\xi)$ with $t> 0$, $x\in \Omega=\R^d$
or $\T^d$, $\xi\in\R^d$. $d\geq 1$ denotes the space dimension.
$\ST=\xi\cdot \na_x-\na_x V \cdot \na_\xi$ is a transport operator.
$V=V(x)$ is a given stationary external forcing. $\SL: L^2_\xi\to
L^2_\xi$ is a linear, local, self-adjoint and non-positive operator
with  ${\rm ker}\,\SL \neq \{0\}$. Moreover,  $\SL$ is degenerately
dissipative in the sense that there is a constant $\la_{\SL}>0$ such
that
\begin{equation}\label{s1.coerc}
    \int_{\R^d} u \SL u\, d\xi\leq -\la_\SL \|\{\SI-\SP\} u\|_{\CH_\xi}^2
\end{equation}
where $\CH_\xi\subset L^2_\xi$ is a Hilbert space with norm $\|\cdot
\|_{\CH_\xi}$, $\SI$ is an identity operator and $\SP$ is an
orthogonal velocity projection operator from $L^2_\xi$ to ${\rm
ker}\,\SL$. Given initial data $u(0,x,\xi)=u_0(x,\xi)$, the solution
of \eqref{s1.eq} formally takes the form of
\begin{equation*}
    u(t)=e^{t\SB}u_0,\ \ \SB=:\SL-\ST.
\end{equation*}
The goal of this paper is to study time-decay rates of $e^{t\SB}u_0$
in some Hilbert space under some conditions on $V$ and $u_0$ as time
tends to infinity.

\bigskip

\noindent{\it 1.2 Literature}. The time rate of convergence of
solutions to equilibrium is an important issue in the study of
evolution equations. For the linear kinetic equation in the form of
\eqref{s1.eq}, the main difficulty lies in the fact that the
linearized operator is degenerate in a nontrivial finite dimensional
space since conservation laws exist for general physical models.
However, the interaction between the transport part and the
degenerate dissipative part can lead to convergence to equilibrium.
This property is called {\it hypocoercivity}  \cite{Vi}.

There have been several well-established methods to study the rate
of convergence for the kinetic models such as the relaxation
equation, Fokker-Planck equation, Boltzmann equation and Landau
equation. Since the literature on this subject is so huge, we only
mention some of them which are related to the study of this paper.
The non-constructive method by spectral analysis to obtain the
exponential rates for the Boltzmann equation with hard potentials on
torus was firstly provided by Ukai \cite{Ukai-1974}. The recent
refinement of results of \cite{Ukai-1974} can be found in
\cite{UY-AA}. Energy method of the Boltzmann equation was developed
by Liu-Yu \cite{Liu-Yu-Shock}, Liu-Yang-Yu \cite{LYY} and Guo
\cite{Guo-CMP,Guo-IUMJ}. Energy method combined with additional
techniques such as velocity weight estimates or spectral analysis
\cite{SG0,SG, DUYZ} was also applied to obtain time decay rates in
the framework of perturbations. Another powerful tool is entropy
method which works in the non-perturbation framework. By using this
method, Desvillettes-Villani \cite{DV} obtained firstly the almost
exponential rate of convergence of solutions to the Boltzmann
equation on torus with soft potentials for large initial data under
the additional regularity conditions. Concerning the Fokker-Planck
equation, see \cite{CCG} and references therein. In addition, on the
basis of the spectral analysis, the hypoelliptic theory  of the
Fokker-Planck equation or relaxation Boltzmann equation  was
developed by H\'{e}rau-Nier \cite{HN} and H\'{e}rau
\cite{H-AA,H-JFA}.

Recently, some general theory on hypocoercivity was provided in
\cite{MN,DMS,Vi}. By constructing some proper Lyapunov functional
defined over the Hilbert space, Mouhot-Neumann \cite{MN} obtained
the exponential rates of convergence in $H^1$-norm for some kinetic
models with general structures in the case of torus. An extension of
\cite{MN} to models in the presence of a confining potential force
was given by Dolbeault-Mouhot-Schmeiser \cite{DMS}, where $L^2$-norm
was considered. Villani \cite{Vi} also gave a systematic study of
the Hypocoercivity theory. The result on the Fokker-Planck in this
paper is inspired by \cite{Vi}.

\bigskip

\noindent{\it 1.3 Notations}.  Define differential operators $\SX_i,
\SY_i$ $(1\leq i\leq d)$ by
\begin{equation}\label{def.op.xy}
    \SX_i u =e^{-\frac{V(x)}{2}}\pa_{x_i} (e^{\frac{V(x)}{2}}
  u),\ \  \SY_i u = e^{-\frac{|\xi|^2}{4}}\pa_{\xi_i} (e^{\frac{|\xi|^2}{4}}
  u).
\end{equation}
Note that $\SX_i,\SY_i $ are equivalent with
\begin{equation*}
 \SX_i  = \frac{1}{2}\pa_{x_i} V +\pa_{x_i} , \ \  \SY_i u = \frac{1}{2}\xi_i  +\pa_{\xi_i}.
\end{equation*}
Define a norm $\|\cdot \|_{\CH^1}$ by
\begin{equation*}
    \|u\|_{\CH^1}^2=\|u\|^2+\|\SX u\|^2+\|\SY u\|^2
\end{equation*}
for $u=u(x,\xi)$. Here and in the sequel, $\|\cdot \|$ means
$L^2$-norm over $\R^d_x\times \R^d_\xi$,
$\SX=(\SX_1,\SX_2,\cdots,\SX_d)$, $\SY=(\SY_1,\SY_2,\cdots,\SY_d)$,
and
\begin{equation*}
\|\SX u\|^2=\sum_{i=1}^d\|\SX_i u\|^2,\ \  \|\SY
u\|^2=\sum_{i=1}^d\|\SY_i u\|^2.
\end{equation*}
For simplicity, when a function under consideration is independent
of velocity variable, $\|\cdot \|$ is also used to denote $L^2$-norm
over $\R^d_x$ and $\pa_i=\pa_{x_i}$ without any confusion. For the
inner products over $L^2_x$ and $L^2_\xi$, we use
\begin{equation*}
   (g,h)=\int_{\Omega} f(x)g(x)dx,\ \  \langle g,h\rangle=\int_{\R^d}
   g(\xi)h(\xi)d\xi,
\end{equation*}
respectively. For a function $w=w(\xi)$, define norms $|\cdot|_w$
and $\|\cdot\|_w$ by
\begin{eqnarray*}
  |g|_{w}^2&=&\int_{\R^d}|\na_\xi g(\xi)|^2 +
  w(\xi)|g(\xi)|^2d\xi,\ \ g=g(\xi),\\
     \|g\|_w^2& =&\iint_{\Omega\times\R^d}|\na_\xi g(x,\xi)|^2 +
   w(\xi)|g(x,\xi)|^2dxd\xi,\ g=g(x,\xi).
\end{eqnarray*}
For $q\geq 1$ and $\Omega=\R^d$, define
\begin{equation*}
    Z_q=L^2_\xi(L^q_x)=L^2(\R^d_\xi;L^q(\R^d_x)),\ \ \|g\|_{Z_q}=\left(\int_{\R^d}\left(\int_{\R^d}
    |g(x,\xi)|^qdx\right)^{2/q}d\xi\right)^{1/2}.
\end{equation*}
Denote two functions $M$, $\CM$ by
\begin{equation*}
     M(\xi)=(2\pi)^{-d/2}e^{-|\xi|^2/2},\ \
     \CM(x,\xi)=e^{-V(x)}M(\xi),
\end{equation*}
where $M$ is a normalized Maxwellian. $C$  denotes some positive
(generally large) constant and $\la$ denotes some positive
(generally small) constant, where both $C$ and $\la$ may take
different values in different places. $A\sim B$ means $\la_1 A\leq B
\leq \la_2 A$ for two generic constants $\la_1>0$ and $\la_2>0$. For
an integrable function $g: \Omega\to\R$, its Fourier transform
$\hat{g}$ is defined by
\begin{equation*}
   \hat{g}(k)= \int_{\Omega} e^{-\rmi x\cdot k} g(x)dx, \ \ x\cdot
    k=:\sum_{j=1}^dx_jk_j,
\end{equation*}
for $k\in\R^d$ if $\Omega=\R^d$ while $k\in \Z^d$ when
$\Omega=\T^d$, where $\rmi =\sqrt{-1}\in \mathbb{C}$ is the
imaginary unit. For two complex numbers $c_1,c_2\in\mathbb{C}$,
$(c_1\mid c_2)=c_1\cdot \overline{c}_2$ denotes the dot product over
the complex field, where $\overline{c}_2$ is the complex conjugate
of $c_2$.

\bigskip

\noindent{\it 1.4 Models}. To the end, we shall consider three types
of degenerately dissipative operators for $\SL$ as follows.

\bigskip

\noindent\underline{Model 1.} $\SL$ is the linear relaxation
operator
\begin{equation}\label{def.lm1}
    \SL u=-\{\SI-\SP_0\}u,
\end{equation}
where $\SP_0: L^2_\xi\to {\rm span}\{M^{1/2}\}$ is an orthogonal
velocity projection operator given by
\begin{equation}\label{def.p0}
    \SP_0 u=a  M^{1/2},\ \ a=a^u=:\langle M^{1/2},
    u\rangle.
\end{equation}
Notice that
\begin{equation*}
{\rm Ker}\, \SL={\rm span}\{M^{1/2}\}
\end{equation*}
and $-\SL$ satisfies the identity
\begin{equation*}
    -\int_{\R^d} u\SL u d\xi=\int_{\R^d}|\{\SI-\SP_0\} u|^2d\xi.
\end{equation*}
Moreover, if
\begin{equation*}
    f=\CM+\CM^{1/2}u,
\end{equation*}
then $f$ equivalently satisfies the original relaxation model:
\begin{equation*}
    \pa_t f+\ST f=M\int_{\R^d}f d\xi-f.
\end{equation*}

\bigskip

\noindent\underline{Model 2.} $\SL$ is the linear self-adjoint
Fokker-Planck operator
\begin{equation}\label{def.lm2}
    \SL u=\frac{1}{M^{1/2}}\na_\xi\cdot \left(M \na_\xi
    \left(\frac{u}{M^{1/2}}\right)\right).
\end{equation}
It is well-known that
\begin{equation*}
{\rm Ker}\, \SL={\rm span}\{M^{1/2}\}
\end{equation*}
and $-\SL$ satisfies the coercivity
\begin{equation}\label{m2.coerc}
    -\int_{\R^d} u\SL u \,d\xi\geq |\{\SI-\SP_0\} u|^2_\nu,
\end{equation}
where $\nu=\nu(\xi)=1+|\xi|^2$ and $\SP_0$ is still defined by
\eqref{def.p0}. Instead of \eqref{m2.coerc}, it is more convenient
to use another equivalent coercivity inequality as used in
\cite{DFT-09}. In fact, define $\SP: L^2_\xi\to {\rm
span}\{M^{1/2},\xi M^{1/2}\}$ as
\begin{equation*}
\left\{\begin{array}{l}
\dis \SP u=\SP_0 u\oplus \SP_1 u,\\[3mm]
\dis \SP_0 u= a M^{1/2},\ \ a=a^u=\langle M^{1/2}, u\rangle,\\[3mm]
\dis \SP_1 u=b\cdot \xi M^{1/2},\ \ b=b^u=\langle \xi M^{1/2},
u\rangle.
\end{array}\right.
\end{equation*}
Then, one can compute
\begin{eqnarray*}
  \SL u &=& \SL\{\SI-\SP\} u+\SL \SP u=\SL\{\SI-\SP\} u+\SL \SP_1 u=
  \SL\{\SI-\SP\} u-\SP_1 u.
\end{eqnarray*}
Thus, it holds that
\begin{equation*}
    -\int_{\R^d} u \SL u\, d\xi\geq \la |\{\SI-\SP\} u|_\nu^2 +|b|^2.
\end{equation*}
Similarly as before, if let
\begin{equation*}
    f= \CM+\CM^{1/2}u
\end{equation*}
then $f$ equivalently satisfies the linear Fokker-Planck equation:
\begin{equation*}
    \pa_t f+\ST f=\na_\xi \cdot (\na_\xi  f+\xi f).
\end{equation*}

\bigskip

\noindent\underline{Model 3.} $\SL$ is the linearized Boltzmann
operator
\begin{equation}\label{def.lm3}
    \SL u=\frac{1}{M^{1/2}} [\SQ(M,M^{1/2}u)+\SQ(M^{1/2}u,M)],
\end{equation}
where $\SQ$ is the so-called bilinear collision operator
 defined by
\begin{eqnarray*}
&\dis \SQ(f,g)=\iint_{\R^d\times S^{d-1}} |(\xi-\xi_\ast)\cdot
\omega| (f(\xi')g(\xi_\ast')-f(\xi)g(\xi_\ast)) d\omega d
\xi_\ast,\\
&\dis \left\{\begin{array}{rll}
              \dis \xi'&=\xi&-[(\xi-\xi_\ast)\cdot\om]\om\\
              \dis \xi_\ast'&=\xi_\ast&+[(\xi-\xi_\ast)\cdot \om]\om
      \end{array}\right.
 \ \ \om\in S^{d-1},
\end{eqnarray*}
for $f=f(\xi)$ and $g=g(\xi)$. Here, although the hard-sphere
collision kernel in $\SQ$ is supposed, all results of this paper
still hold in the case of hard potentials and Maxwell molecules. For
$\SL$, it is also well-known \cite{CIP-Book} that ${\rm dim}\,{\rm
ker}\, \SL=d+2$,
\begin{equation}\label{def.ker.BE}
    {\rm ker}\,\SL={\rm span}\{M^{1/2}, \xi_1M^{1/2},\cdots, \xi_d M^{1/2},
    |\xi|^2M^{1/2}\},
\end{equation}
and $-\SL$ satisfies the coercivity
\begin{equation*}
    -\int_{\R^d} u \SL u\, d\xi \geq \la \int_{\R^d} \nu(\xi)|\{\SI-\SP\}
    u|^2 d\xi,
\end{equation*}
where we still used $\nu(\xi)$ to denote the collision frequency
defined by
\begin{equation}\label{def.nu.m3}
    \nu(\xi)=\iint_{\R^d\times S^{d-1}}|(\xi-\xi_\ast)\cdot \omega|M
    d\omega d\xi_\ast,
\end{equation}
and also used $\SP:L^2_\xi \to {\rm ker}\,\SL$ to denote the
orthogonal velocity projection operator. For convenience,
corresponding to the $d+2$ dimensional space \eqref{def.ker.BE},
$\SP$ is written as
\begin{equation}\label{def.pm3}
\left\{\begin{array}{l}
\dis \SP u =\{a+b\cdot \xi+c(|\xi|^2-d)\}M^{1/2},\\[3mm]
\dis a=a^u=\langle M^{1/2}, u\rangle,\\[3mm]
\dis b=b^u=\langle \xi M^{1/2}, u\rangle,\\[3mm]
\dis c=c^u=\frac{1}{2d}\langle (|\xi|^2-d)M^{1/2}, u\rangle.
\end{array}\right.
\end{equation}
Therefore, $a,b,c$ mean mass, momentum and temperature components of
macroscopic part $\SP u$. If let $f=\CM+\CM^{1/2} u$ then $f$
satisfies the linear Boltzmann equation:
\begin{equation}\label{def.lm3-1}
    \pa_t f+\ST f = \SQ(M,f)+\SQ(f,M).
\end{equation}

\bigskip

\noindent{\it 1.5 Main results.} Let us state them in two cases
which will be proved in terms of different analytical tools.

\begin{theorem}[case of no force]\label{thm.nf}
Consider \eqref{s1.eq} where $\Omega=\R^d$ or $\T^d$, $d\geq 1$,
$V=0$, and $\SL$ is one of the linear relaxation operator, linear
Fokker-Planck operator and linearized Boltzmann operator as defined
in Model 1, Model 2 and Model 3, respectively. Let $e^{t\SB} u_0$
denote the corresponding solution for initial data $u_0=u_0(x,\xi)$.

\medskip

\noindent Case ($\Omega=\R^d$). Let $h=h(t,x,\xi)$ satisfy
\begin{equation}\label{thm.nf.ch}
    h(t,x,\xi)\perp {\rm ker}\,\SL,\ \ \forall\,t\geq 0,x\in \R^d.
\end{equation}
Let $1\leq q\leq 2$. Then, for any $\al$, $\al'$ with $\al'\leq \al$
and $m=|\al-\al'|$, there is a constant $C$ such that
\begin{equation}\label{thm.nf.1}
    \|\pa_x^\al e^{t\SB} u_0\|\leq C (1+t)^{-\si_{q,m}} (\|\pa_x^{\al'}u_0\|_{Z_q}+\|\pa_x^\al
    u_0\|)
\end{equation}
and
\begin{eqnarray}
&\dis \|\pa_x^\al \int_0^t e^{-(t-s)\SB} h(s)ds\|^2\nonumber \\
&\dis \leq C\int_0^t (1+t-s)^{-2\si_{q,m}} (\|w^{-1/2}\pa_x^{\al'}
h(s)\|_{Z_q}^2+\|w^{-1/2}\pa_x^{\al}h(s)\|^2)ds\label{thm.nf.2}
\end{eqnarray}
for any $t\geq 0$. Here, $w=w(\xi)$ is defined by
\begin{equation*}
    w=w(\xi)=\left\{\begin{array}{ll}
               1 & \ \ \text{for Model 1,}\\
               1+|\xi|^2  & \ \ \text{for Model 2,}\\
                1+|\xi| & \ \ \text{for Model 3,}\\
             \end{array}\right.
\end{equation*}
and the index $\si_{a,m}$ of the algebraic rate is defined by
\begin{equation*}
    \si_{q,m}=\frac{d}{2}\left(\frac{1}{q}-\frac{1}{2}\right)+\frac{m}{2}.
\end{equation*}

\medskip

\noindent Case ($\Omega=\T^d$). Suppose
\begin{equation}\label{thm.nf.cid}
    \iint_{\T^d\times\R^d} \psi (\xi)u_0(x,\xi)dxd\xi=0,\ \ \forall\,
    \psi\in {\rm ker}\,\SL.
\end{equation}
Then, there are constants $C$ and $\la>0$ such that
\begin{equation}\label{thm.nf.3}
    \|e^{t\SB} u_0\|\leq C e^{-\la t}\|u_0\|
\end{equation}
for any $t\geq 0$.

\end{theorem}

\begin{theorem}[case of confining force]\label{thm.f}
Consider \eqref{s1.eq} where $\Omega=\R^d$, $d\geq 1$, $V=V(x)$ is a
confining force with
\begin{equation*}
    \int_{\R^d} e^{-V(x)}dx=1,
\end{equation*}
 $\SL$ is one of the linear relaxation operator, linear
Fokker-Planck operator and linearized Boltzmann operator as defined
in Model 1, Model 2 and Model 3, respectively. Let $e^{t\SB} u_0$
denote the corresponding solution for initial data $u_0=u_0(x,\xi)$.
The following additional conditions on $V$ and $u_0$ are supposed to
hold.

\medskip

\noindent{Case of Model 1:} $V=\frac{|x|^2}{2}-\frac{d}{2}\ln
(2\pi)$, and
\begin{equation*}
    \iint_{\R^d\times\R^d} \CM^{1/2} u_0 dxd\xi=0.
\end{equation*}

\medskip

\noindent{Case of Model 2:}
\begin{eqnarray*}
&\dis \frac{1}{4}|\na_x V|^2-\frac{1}{2}\De_x V\to \infty\ \
\text{as}\ |x|\to \infty,\\
&\dis |\na_x^2 V|^2\leq \delta |\na_x V|^2+C_\delta, \ \ \forall\,
\delta>0,
\end{eqnarray*}
and
\begin{equation*}
    \iint_{\R^d\times\R^d} \CM^{1/2} u_0 dxd\xi=0.
\end{equation*}

\medskip

\noindent{Case of Model 3:} $V=\frac{|x|^2}{2}-\frac{d}{2}\ln
(2\pi)$, $d\geq 3$, and
\begin{equation*}
\iint_{\R^d\times\R^d} (1,x,\xi,x\cdot \xi,x\times \xi,
|x|^2,|\xi|^2)\CM^{1/2} u_0 dxd\xi=0,
\end{equation*}
where $(x\times\xi)_{ij}=x_i\xi_j-x_j\xi_i$, $1\leq i,j\leq d$.

\medskip

\noindent Then, under the above assumptions, there are constants $C$
and $\la>0$ such that
\begin{equation}\label{thm.f.1}
    \|e^{t\SB}\|_{\CH^1}\leq C e^{-\la t}\|u_0\|_{\CH^1}
\end{equation}
for any $t\geq 0$.
\end{theorem}

\bigskip

\noindent{\it 1.6 Strategy of proof.} The main idea is the
macro-micro decomposition combined with Kawashima's argument on
dissipation of the hyperbolic-parabolic system. In fact, suppose
that ${\rm ker}\,\SL$ is spanned by an orthogonal set
\begin{equation*}
    \CA= \{\psi_0(\xi),\psi_1(\xi),\cdots,\psi_n(\xi)\}
\end{equation*}
and the corresponding orthogonal velocity projection is denoted by
\begin{equation*}
   \SP u=\sum_{i=0}^n a_i(t,x)\psi_i(\xi).
\end{equation*}
The total energy dissipation rate corresponding to certain temporal
energy functional of $u=\{\SI-\SP\}u+\SP u$ can be recovered as
follows:

\begin{itemize}
  \item[\underline{Step 1}.] Starting from the equation
  \eqref{s1.eq}, one can make energy estimates to obtain the
  dissipation of the microscopic part $\{\SI-\SP\}u$ on the basis of
  the coercivity property \eqref{s1.coerc} of the linearized operator
  $\SL$.
\item[\underline{Step 2}.] One can derive a fluid-type one-order hyperbolic
system of $a_i(t,x)$ $(0\leq i\leq n)$ coupled with $\{\SI-\SP\} u$
which are actually of the moment equations in terms of the above
orthogonal set $\CA$ and some high-order moment functions. See
\eqref{s2.w.ab}, \eqref{s2.f.p3}, \eqref{s4.nf.fs} and
\eqref{s4.f.fs} for models under consideration. By applying
Kawashima's argument on dissipation of the mixed
hyperbolic-parabolic system \cite{Ka}, one can obtain the
dissipation of the macroscopic part $\SP u$ or equivalently
$a_i(t,x)$ $(0\leq i\leq n)$ on the basis of the fluid-type system.
\item[\underline{Step 3}.] Combining estimates in Step 1 and Step 2,
one can obtain a properly defined temporal Lyapunov functional which
is not only  equivalent with the desired total energy functional but
also captures the total energy dissipation rate.
\end{itemize}

Analytical tools are the Fourier transform for the case when there
is no forcing and the direct energy estimates otherwise. When there
is a potential force, we also need the Poincar\'{e} inequality and
Korn-type inequality which will be provided at the last section.

\section{Relaxation model}

In this section we prove Theorem \ref{thm.nf} and Theorem
\ref{thm.f} for Model 1.

\bigskip

\noindent{\it 2.1} {\it Case when $\Omega=\R^d$ and $V=0$}. To prove
\eqref{thm.nf.1} and \eqref{thm.nf.2}, we consider the Cauchy
problem
\begin{equation}\label{s2.e}
\left\{\begin{array}{rll}
 \dis \pa_t u+\xi \cdot \na_x u &= \SL u+ h, &\ \ t> 0,x\in
  \R^d,\xi\in \R^d,\\[2mm]
 \dis  u(0,x,\xi)&=u_0(x,\xi),&\ \ x\in
  \R^d,\xi\in \R^d,
\end{array}\right.
\end{equation}
where $\SL$ is the relaxation operator given in \eqref{def.lm1}, and
as in \eqref{thm.nf.ch}, $h=h(t,x,\xi)$ satisfies
\begin{equation}\label{s2.con.h}
    h(t,x,\xi)\perp {\rm ker}\,\SL,\ \ \forall\, t\geq 0,x\in \R^d.
\end{equation}
Notice that the solution to \eqref{s2.e}  can be written as
\begin{equation}\label{s2.sol}
    u(t)=e^{t\SB}u_0 +\int_0^t e^{-(t-s)\SB} h(s)ds,
\end{equation}
where $\SB=-\xi\cdot \na_x+\SL$.

We shall use the method of Fourier transform to deal with the
time-decay of the solution $u$ given by \eqref{s2.sol} in a unifying
manner so that \eqref{thm.nf.1} and \eqref{thm.nf.2} follow from the
case when $h=0$ and $u_0=0$, respectively. In fact, the Fourier
transform of \eqref{s2.e}$_1$ shows
\begin{equation*}
    \pa_t \hat{u}+\rmi \xi\cdot k \hat{u}=\SL \hat{u}+ \hat{h},
\end{equation*}
which after multiplying $\bar{\hat{u}}$ and taking velocity
integration and the real part, deduces
\begin{equation*}
    \frac{1}{2}\pa_t
    \|\hat{u}\|_{L^2_\xi}^2+\|\{\SI-\SP_0\}\hat{u}\|_{L^2_\xi}^2=\rmre
    \int_{\R^d}(\hat{h}\mid \hat{u})d\xi.
\end{equation*}
Due to \eqref{s2.con.h},
\begin{eqnarray*}
  \int_{\R^d}(\hat{h}\mid \hat{u})d\xi  &=&  \int_{\R^d}(\hat{h}\mid
  \{\SI-\SP_0\}\hat{u})d\xi
  \leq  \frac{1}{2}\|\hat{h}\|_{L^2_\xi}^2 +\frac{1}{2}
  \|\{\SI-\SP_0\}\hat{u}\|_{L^2_\xi}^2.
\end{eqnarray*}
Then, it follows that
\begin{equation}\label{s2.w.p1}
\pa_t
    \|\hat{u}\|_{L^2_\xi}^2+\|\{\SI-\SP_0\}\hat{u}\|_{L^2_\xi}^2\leq
    \|\hat{h}\|_{L^2_\xi}^2,
\end{equation}
which is the first estimate on the basis of the dissipative property
of $\SL$.  Next, we estimate $\SP_0 u=a M^{1/2}$ with $a=a^u=\langle
M^{1/2},u\rangle$. Recall $b=b^u=\langle\xi M^{1/2},u\rangle$. Then,
from \eqref{s2.e}$_1$, $a$ and $b$ satisfy  the fluid-type system
\begin{equation}\label{s2.w.ab}
\left\{\begin{array}{l}
\dis \pa_t a +\na_x \cdot b=0,\\[3mm]
\dis \pa_t b + \na_x a + \na_x \cdot \Ga (\{\SI-\SP_0\}u)=-b+
\langle \xi M^{1/2}, h\rangle,
\end{array}\right.
\end{equation}
where $\Ga=(\Ga_{ij})_{d\times d}$ is the moment function defined by
\begin{equation}\label{def.Ga}
    \Ga_{ij}(g)=\langle (\xi_i\xi_j-1)M^{1/2}, g\rangle,\ \ 1\leq
    i,j\leq d.
\end{equation}
Notice by the definition of $\SP_0$ that
\begin{equation*}
    \Ga(\{\SI-\SP_0\}u)=\langle \xi\otimes\xi M^{1/2},
    \{\SI-\SP_0\}u\rangle.
\end{equation*}
Here, we used \eqref{def.Ga} as the definition of $\Ga$ only for the
convenience of the later proof for Model 3. Taking further the
Fourier transform of \eqref{s2.w.ab} gives
\begin{equation*}
\left\{\begin{array}{l}
\dis \pa_t \hat{a} + \rmi k\cdot \hat{b}=0,\\[3mm]
\dis \pa_t \hat{b}+\rmi k \hat{a} +\rmi
\Ga(\{\SI-\SP_0\}\hat{u})\cdot k=-\hat{b}+\langle \xi M^{1/2},
\hat{h}\rangle
\end{array}\right.
\end{equation*}
where
\begin{equation*}
\left(\Ga(\{\SI-\SP_0\}\hat{u})\cdot
k\right)_i=\sum_{j=1}^d\Ga_{ij}(\{\SI-\SP_0\}\hat{u}) k_j.
\end{equation*}
For the above fluid-type system, from computations
\begin{eqnarray*}
|k|^2|\hat{a}|^2&=& ( \rmi k \hat{a}\mid \rmi k \hat{a}) =( \rmi k
\hat{a}\mid -\pa_t \hat{b}-\rmi
\Ga(\{\SI-\SP_0\}\hat{u})\cdot k-\hat{b}+\langle \xi M^{1/2},\hat{h}\rangle)\\
&=&-\pa_t (\rmi \hat{a}\mid \hat{b}) + (\rmi k \pa_t \hat{a}\mid
\hat{b})\\
&&+ ( \rmi k \hat{a}\mid -\rmi \Ga(\{\SI-\SP_0\}\hat{u})\cdot
k-\hat{b}+\langle \xi M^{1/2},\hat{h}\rangle),
\end{eqnarray*}
one has
\begin{eqnarray*}
&\dis \pa_t \rmre (\rmi \hat{a}\mid \hat{b})  +
|k|^2|\hat{a}|^2\\
&\dis \leq |k\cdot \hat{b}|^2 +\eps |k|^2 |\hat{a}|^2
+C_\eps(1+|k|^2)\|\{\SI-\SP_0\}\hat{u}\|_{L^2_\xi}^2 + C_\eps
\|\hat{h}\|_{L^2_\xi}^2,
\end{eqnarray*}
where $\eps>0$ is arbitrary. Then, taking $\eps>0$ small and
dividing it by $1+|k|^2$ yield
\begin{equation}\label{s2.w.p2}
    \pa_t \rmre \frac{ (\rmi \hat{a}\mid \hat{b})
    }{1+|k|^2}+\frac{\la |k|^2}{1+|k|^2}|\hat{a}|^2 \leq C
    (\|\{\SI-\SP_0\}\hat{u}\|_{L^2_\xi}^2+\|\hat{h}\|_{L^2_\xi}^2),
\end{equation}
which is the second estimate based on the Kawashima's argument on
the dissipation of the hyperbolic-parabolic system.

Now, for $t\geq 0$, $k\in \R^d$, define
\begin{equation}\label{s2.def.E}
    E(\hat{u})=\|\hat{u}\|_{L^2_\xi}^2+ \kappa \rmre \frac{ (\rmi \hat{a}\mid \hat{b})
    }{1+|k|^2}
\end{equation}
with a small constant $\kappa>0$ to be determined later. One can let
$\kappa>0$ be small such that $E(\hat{u})\sim
\|\hat{u}\|_{L^2_\xi}^2 $. Taking $\kappa>0$ further small, the
linear combination of \eqref{s2.w.p1} and \eqref{s2.w.p2} gives
\begin{equation*}
    \pa_t E(\hat{u}) + \frac{\la |k|^2}{1+|k|^2} E(\hat{u})\leq
    C\|\hat{h}\|_{L^2_\xi}^2,
\end{equation*}
which with the help of Gronwall's inequality implies
\begin{eqnarray*}
  E(\hat{u}(t,k))\leq e^{-\frac{\la |k|^2 }{1+|k|^2}t}  E(\hat{u}_0(k))+C\int_0^t
  e^{-\frac{\la |k|^2 }{1+|k|^2}(t-s)}\|\hat{h}(s,k)\|_{L^2_\xi}^2
  ds.
\end{eqnarray*}
Hence, \eqref{thm.nf.1} and \eqref{thm.nf.2} follows from the above
estimate by using the standard procedure as in \cite{Ka,Gl,UY-AA},
and the details of the rest proof are omitted for simplicity. This
completes the proof of \eqref{thm.nf.1} and \eqref{thm.nf.2} for
Model 1 in Theorem \ref{thm.nf}.

\bigskip

\noindent{\it 2.2} {\it Case when $\Omega=\T^d$ and $V=0$}. The goal
in this case is to prove \eqref{thm.nf.3}. It actually can be
achieved by a little modification in the case of the whole space. In
fact, let's consider the Cauchy problem \eqref{s2.e} with the
spatial domain $\R^d$ replaced by $\T^d$ and $h=0$. Then,
\eqref{s2.w.p1} and \eqref{s2.w.p2} with $h=0$ still hold for
$\Omega=\T^d$. $a$ and $b$ still satisfy the system \eqref{s2.w.ab}.
Notice that one has the mass conservation
\begin{equation*}
    \frac{d}{dt}\int_{\T^d} a(t,x)dx=0.
\end{equation*}
Since initially
\begin{equation*}
    \iint_{\T^d\times \R^d}M^{1/2} u_0(x,\xi) dxd\xi=\int_{\T^d}
    a_0(x)dx=0
\end{equation*}
which corresponds to the assumption \eqref{thm.nf.cid}, then
\begin{equation*}
\left.\int_{\T^d} a(t,x)dx\right|_{t\geq 0}\equiv 0.
\end{equation*}
Thus, $\hat{a}(t,0)=0$ for any $t\geq 0$, which yields
\begin{equation*}
    \frac{|k|^2}{1+|k|^2} |\hat{a}|^2\geq \frac{1}{2}|\hat{a}|^2\ \ \ \text{for any}\  t\geq 0, k\in \Z^d.
\end{equation*}
Therefore, \eqref{s2.w.p2} is modified as
\begin{equation*}
    \pa_t \rmre \frac{ (\rmi \hat{a}\mid \hat{b})
    }{1+|k|^2}+\la |\hat{a}|^2 \leq C
    \|\{\SI-\SP_0\}\hat{u}\|_{L^2_\xi}^2.
\end{equation*}
By using the same definition of $ E(\hat{u})$ as in
\eqref{s2.def.E}, it holds
\begin{equation*}
    \pa_t E(\hat{u})+\la E (\hat{u})\leq 0,
\end{equation*}
which implies \eqref{thm.nf.3} from Gronwall's inequality and
$k$-integration. This completes the proof of \eqref{thm.nf.3} for
Model 1 in Theorem \ref{thm.nf}.

\bigskip

\noindent{\it 2.3} {\it Case when $\Omega=\R^d$ and $V$ is
confining}. In particular, let $V=\frac{|x|^2}{2}-\frac{d}{2}\ln
(2\pi)$. Take $u_0$ with
\begin{equation}\label{s2.f.id}
    \iint_{\T^d\times \R^d} \CM^{1/2} u_0 dxd\xi=0.
\end{equation}
Let $u(t)= e^{t\SB} u_0$ be the solution to the Cauchy problem
\begin{equation}\label{s2.f.e}
    \left\{\begin{array}{rll}
      \dis \pa_t u+ \ST u &=\SL u, &\ \ t>0,x\in\R^d,\xi \in \R^d,\\[2mm]
     \dis  u(0,x,\xi)&=u_0(x,\xi), &\ \ x\in\R^d,\xi \in \R^d.
    \end{array}\right.
\end{equation}

Next, we make energy estimates on $u$. For zero-order, it is
straightforward to get
\begin{equation}\label{s2.f.p0}
    \frac{1}{2}\frac{d}{dt}\|u\|^2+\|\{\SI-\SP_0\}u\|^2=0.
\end{equation}
For first-order, instead of directly estimating $x$ and $\xi$
derivatives, we use $\SX$ and $\SY$ differentiations. Take $1\leq
i\leq d$. Applying $\SX_i$ and $\SY_i$ to \eqref{s2.f.e}$_1$, one
has
\begin{eqnarray}
\pa_t \SX_i u+ \ST \SX_i u-\SL\SX_i u&=& [\ST,\SX_i]u-[\SL,\SX_i]u,\label{s2.f.p1}\\
\pa_t \SY_i u+ \ST \SY_i u-\SL\SY_i u&=&
[\ST,\SY_i]u-[\SL,\SY_i]u,\label{s2.f.p2}
\end{eqnarray}
where $[\cdot,\cdot]$ denotes commutator of two operators. Since
$\SL$ is local in $x$ and $\SX_i$ is only involved in spatial
derivative and multiplier, then $[\SL,\SX_i]=0$. Due to the fact
that $\SP_0\SY_i=\SY_i\SP_0=0$,
\begin{equation*}
[\SL,\SY_i]=[\SI-\SP_0,\SY_i]=-[\SP_0,\SY_i]=0.
\end{equation*}
For commutators containing $\ST$, the further computations show that
\begin{eqnarray}
 [\ST,\SX_i]u &=& [\xi\cdot \na_x,\SX_i] u- [\na_x V\cdot \na_\xi,\SX_i]
 u\nonumber \\
 &=&\frac{1}{2}\xi\cdot\na_x \pa_iV u +\na_x \pa_iV\cdot \na_\xi u=\na_x
 \pa_iV\cdot \SY u,\label{s2.f.p2-0}
\end{eqnarray}
and
\begin{eqnarray}
 [\ST,\SY_i]u &=& [\xi\cdot \na_x,\SY_i] u- [\na_x V\cdot \na_\xi,\SY_i]
 u \nonumber \\
 &=&-\pa_i u-\frac{1}{2}\pa_i V u=-\SX_i u.\label{s2.f.p2-1}
\end{eqnarray}
Thus, it follows from \eqref{s2.f.p1} and \eqref{s2.f.p2} that
\begin{eqnarray*}
&&\frac{1}{2}\frac{d}{dt}\|\SX u\|^2+\|\{\SI-\SP_0\}\SX
u\|^2=\sum_{ij=1}^d\iint_{\R^d\times\R^d}\pa_i\pa_jV \SX_i u\SY_j u
dxd\xi,\\
&&\frac{1}{2}\frac{d}{dt}\|\SY u\|^2+\|\{\SI-\SP_0\}\SY
u\|^2=-\sum_{ij=1}^d\iint_{\R^d\times\R^d}\de_{ij} \SX_i u\SY_j u
dxd\xi,
\end{eqnarray*}
where $\de_{ij}$ is the Kronecker delta. Noticing $\pa_i\pa_j
V=\de_{ij}$ by the definition of $V$ in the considered case, one has
\begin{eqnarray}
&&\frac{1}{2}\frac{d}{dt}(\|\SX u\|^2+\|\SY
u\|^2)+\|\{\SI-\SP_0\}\SX u\|^2+\|\{\SI-\SP_0\}\SY
u\|^2=0.\label{s2.f.p0-1}
\end{eqnarray}
Since $\SP_0 \SY=0$, the rest is to obtain the dissipation rate
corresponding to $\SP_0\SX u=\SX \SP_0 u$ and $\SP_0 u$, or
equivalently $\SX a$ and $a$. We shall again turn to the fluid-type
system satisfied by $a$ and $b$.

Notice that when $V$ is nontrivial, similarly as before, from
\eqref{s2.f.e}$_1$, $a$ and $b$ satisfy
\begin{equation}\label{s2.f.p0-2}
    \left\{\begin{array}{l}
     \dis \pa_t (a e^{-\frac{V}{2}})+\na_x \cdot (b
     e^{-\frac{V}{2}})=0,\\[3mm]
     \dis \pa_t (b  e^{-\frac{V}{2}})+\na_x (a
     e^{-\frac{V}{2}})+\na_x V a  e^{-\frac{V}{2}}+\na_x\cdot [\Ga(\{\SI-\SP_0\}u)
e^{-\frac{V}{2}}]=-be^{-\frac{V}{2}}.
    \end{array}\right.
\end{equation}
Equivalently, the above system can be rewritten as
\begin{equation}\label{s2.f.p3}
    \left\{\begin{array}{l}
     \dis \pa_t a -\SX^\ast\cdot b=0,\\[3mm]
     \dis \pa_t b+\SX a -\SX^\ast\cdot  \Ga(\{\SI-\SP_0\}u)+b=0,
    \end{array}\right.
\end{equation}
where $\SX^\ast$ is the adjoint operator of $\SX$ given by
\begin{equation*}
    \SX^\ast_i u=-e^{\frac{V(x)}{2}} \pa_{x_i} (e^{-\frac{V(x)}{2}}
    u)=(\frac{1}{2}\pa_{x_i} V -\pa_{x_i})u,\ \ 1\leq i\leq d.
\end{equation*}
Noticing that \eqref{s2.f.id} implies
\begin{equation*}
    \int_{\R^d} e^{-\frac{V(x)}{2}}a(0,x) dx=0,
\end{equation*}
it follows from the mass conservation \eqref{s2.f.p0-2}$_1$ that
\begin{equation}\label{s2.f.p0-3}
    \left.\int_{\R^d} e^{-\frac{V(x)}{2}}a(t,x) dx\right|_{t\geq
    0}=0.
\end{equation}
From \eqref{s2.f.p3}$_2$, one can compute
\begin{eqnarray*}
  \|\SX a\|^2 &=& (\SX a, \SX a)
  =(\SX a, -\pa_t b+\SX^\ast\cdot  \Ga(\{\SI-\SP_0\}u)-b)\\
  &=&-\frac{d}{dt} (\SX a, b) +(\SX \pa_t a, b) +(\SX a, \SX^\ast\cdot
  \Ga(\{\SI-\SP_0\}u)-b)
\end{eqnarray*}
where it further holds from \eqref{s2.f.p3}$_1$ that
\begin{equation*}
(\SX \pa_t a, b)=(\pa_t a, \SX^\ast\cdot b)=(\SX^\ast\cdot b,
\SX^\ast\cdot b)=\|\SX^\ast\cdot b\|^2.
\end{equation*}
Then, it follows that
\begin{equation}\label{s2.f.p4}
    \frac{d}{dt} (\SX a, b)+\la   \|\SX a\|^2 \leq \|\SX^\ast\cdot
    b\|^2+C(\|\SX^\ast\{\SI-\SP_0\}u\|^2+\|b\|^2).
\end{equation}

\begin{lemma}\label{lem.ineq.V}
As long as there is a constant $C$ such that
\begin{equation}\label{lem.ineq.V.1}
    |\De_x V|^2\leq C(|\na_x V|^2 +1)
\end{equation}
for all $x\in\R^d$, there is some constant $C$ such that
\begin{equation}\label{lem.ineq.V.2}
    \int_{\R^d}|\SX^\ast g|^2 dx\leq C\int_{\R^d}(|\SX g|^2+|g|^2
    )dx
\end{equation}
for $g=g(x)$.
\end{lemma}

\begin{proof}
From integration by parts,
\begin{equation}\label{lem.ineq.V.p1}
    \int_{\R^d} |\SX^\ast g|^2dx=\int_{\R^d} |\SX g|^2 dx
    +\int_{\R^d} |g|^2 \De_x V dx.
\end{equation}
Under the assumption \eqref{lem.ineq.V.1},
\begin{eqnarray*}
 \int_{\R^d} |g|^2 \De_x V dx &\leq & \int_{\R^d} |g|^2 (\eps |\De_x V|^2+\frac{1}{4\eps})
 dx\\
 &\leq & \int_{\R^d} |g|^2 (C\eps|\na_x V|^2+C\eps+\frac{1}{4\eps})
 dx\\
 &=&4C\eps\int_{\R^d} |g|^2 (\frac{1}{4}|\na_x V|^2-\frac{1}{2}\De_x V) dx +2C\eps
 \int_{\R^d} |g|^2 \De_x V dx\\
 &&+ (C\eps +\frac{1}{4\eps})\int_{\R^d}|g|^2dx
\end{eqnarray*}
for arbitrary $\eps>0$. Taking $\eps>0$ small, one has
\begin{equation*}
 \int_{\R^d} |g|^2 \De_x V dx\leq C  \int_{\R^d} |\SX g|^2
 dx+C\int_{\R^d}|g|^2dx.
\end{equation*}
Then, \eqref{lem.ineq.V.2} follows by plugging the above inequality
into \eqref{lem.ineq.V.p1}.\end{proof}

So, by applying Lemma \ref{lem.ineq.V} to \eqref{s2.f.p4}, one has
\begin{equation*}
    \frac{d}{dt} (\SX a, b)+\la   \|\SX a\|^2 \leq C(\|\{\SI-\SP_0\}\SX u\|^2+\|\{\SI-\SP_0\}u\|^2).
\end{equation*}
Furthermore, due to \eqref{s2.f.p0-3}, one has Poincar\'{e}
inequality
\begin{equation*}
    \|\SX a\|^2\geq \la \|a\|^2
\end{equation*}
from Proposition \ref{prop.pn}. Then, it follows that
\begin{equation}\label{s2.f.p5}
    \frac{d}{dt} (\SX a, b)+\la  ( \|\SX a\|^2+\|a\|^2) \leq C(\|\{\SI-\SP_0\}\SX u\|^2
    +\|\{\SI-\SP_0\}u\|^2).
\end{equation}

Now, let us define a temporal functional
\begin{equation*}
    \CE(u(t))=\|u\|^2+\|\SX u\|^2+\|\SY u\|^2+ \kappa (\SX a, b)
\end{equation*}
with a small constant $\kappa>0$ to be determined. Firstly,
$\kappa>0$ is chosen small such that
\begin{equation*}
 \CE(u(t))\sim \|u(t)\|_{\CH^1}^2.
\end{equation*}
$\kappa>0$ is further small enough such that the linear combination
of \eqref{s2.f.p0}, \eqref{s2.f.p0-1} and \eqref{s2.f.p5} gives
\begin{equation*}
    \frac{d}{dt}\CE(u(t))+\la \CE(u(t))\leq 0.
\end{equation*}
By Gronwall's inequality, this proves \eqref{thm.f.1} and hence
completes the proof of Theorem \ref{thm.f} for Model 1.

\section{Fokker-Planck equation}

In this section we prove Theorem \ref{thm.nf} and Theorem
\ref{thm.f} for Model 2. The proof in the case of no forcing can be
carried out in the completely same way as for Model 1. When a
stationary potential forcing is present, the proof for  Model 1 can
be refined to yield the exponential time-decay rate in $\CH^1$ even
for a class of  general potential functions essentially because the
Fokker-Planck operator enjoys the velocity regularity. This actually
has been studied in detail by Villani \cite{Vi}, and the
hypoelliptic theory is founded by  H\'{e}rau \cite{H-AA,H-JFA}.
Here, we shall give another proof which is based on the macro-micro
decomposition and Kawashima's dissipation argument on the
hyperbolic-parabolic system.

\bigskip

\noindent{\it 3.1} {\it Case when $\Omega=\R^d$ or $\T^d$ and
$V=0$}. We consider the Cauchy problem
\begin{equation}\label{s3.e}
\left\{\begin{array}{rll}
\dis \pa_t u+\xi \cdot \na_x u &= \SL u+
h, &\ \ t> 0,x\in
\Omega,\xi\in \R^d,\\[2mm]
\dis  u(0,x,\xi)&=u_0(x,\xi), &\ \ x\in
 \Omega,\xi\in \R^d,
\end{array}\right.
\end{equation}
where $\SL$ is the self-adjoint Fokker-Planck operator given in
\eqref{def.lm2}, and $h=h(t,x,\xi)$ satisfies
\begin{equation*}
    h(t,x,\xi)\perp {\rm ker}\,\SL,\ \ \forall\, t\geq 0,x\in \Omega.
\end{equation*}
As before, the solution to \eqref{s3.e} can be written as
\begin{equation*}
    u(t)=e^{t\SB}u_0 +\int_0^t e^{-(t-s)\SB} h(s)ds,
\end{equation*}
where $\SB=-\xi\cdot \na_x+\SL$. In this case,  Theorem \ref{thm.nf}
and Theorem \ref{thm.f} for Model 2 can be proved in the same way as
for Model 1. Thus, all details of the proof in this case are omitted
for simplicity.

\bigskip

\noindent{\it 3.2} {\it Case when $\Omega=\R^d$ and $V$ is
confining}. Suppose
\begin{eqnarray}
&\dis  \int_{\R^d} e^{-V(x)}dx=1,\label{s3.v.c1}\\
&\dis \frac{1}{4}|\na_x V|^2-\frac{1}{2}\De_x V\to \infty \ \
\text{as}\ \ |x|\to \infty,\label{s3.v.c2}\\
&\dis |\na_x^2 V|\leq \delta |\na_x V|^2 +C_\delta,\ \
\forall\,\delta>0.\label{s3.v.c3}
\end{eqnarray}
Take $u_0$ with
\begin{equation}\label{s3.v.c4}
    \iint_{\T^d\times \R^d}\CM^{1/2} u_0 dxd\xi=0.
\end{equation}
Let $u(t)= e^{t\SB} u_0$ be the solutio to the Cauchy problem
\begin{equation}\label{s3.f.e}
    \left\{\begin{array}{rll}
      \dis \pa_t u+ \ST u &=\SL u, &\ \ t>0,x\in\R^d,\xi \in \R^d,\\[2mm]
     \dis  u(0,x,\xi)&=u_0(x,\xi), &\ \ x\in\R^d,\xi \in \R^d.
    \end{array}\right.
\end{equation}

Next, we follow the same line of proof as for Model 1. It is
straightforward to get from \eqref{s3.f.e}$_1$ that
\begin{equation}\label{s3.f.p0}
    \frac{1}{2} \frac{d}{dt}\|u\|^2 +\la\|\{\SI-\SP_0\}
    u\|_\nu^2\leq 0,
\end{equation}
where $\nu=\nu(\xi)=1+|\xi|^2$. Recall the definition
\eqref{def.op.xy} of the operator $\SY$. Observe that since $\SL u=
\De_\xi u +\frac{1}{4} (2d-|\xi|^2)u$, then $[\SL, \SY]=\SY$. Thus,
similar to obtain \eqref{s2.f.p1}-\eqref{s2.f.p2} with identities
\eqref{s2.f.p2-0}-\eqref{s2.f.p2-1}, one has
\begin{eqnarray*}
  \pa_t \SX_i u +\ST \SX_i u-\SL \SX_i u &=& \na_x \pa_i V\cdot \SY u,\\
 \pa_t \SY_i u +\ST \SY_i u-\SL \SY_i u &=& -\SX_i u-\SY_i u,
\end{eqnarray*}
which imply
\begin{eqnarray}
  &\dis \frac{1}{2}\frac{d}{dt}\|\SX u\|^2 +\la \|\{\SI-\SP_0\} \SX u\|_\nu^2 \leq
  \sum_{ij=1}^d\iint_{\R^d\times \R^d}\pa_i\pa_j V \SX_i u \SY_j u
  dxd\xi,\label{s3.f.p1}\\
  &\dis \frac{1}{2}\frac{d}{dt}\|\SY u\|^2 +\la \|\{\SI-\SP_0\} \SY u\|_\nu^2 +\|\SY u\|^2\leq
  -\sum_{i=1}^d\iint_{\R^d\times \R^d}\SX_i u \SY_j u
  dxd\xi.\label{s3.f.p2}
\end{eqnarray}
The r.h.s. of \eqref{s3.f.p1} is bounded as
\begin{eqnarray*}
&& \sum_{ij=1}^d\iint_{\R^d\times \R^d}\pa_i\pa_j V \SX_i u \SY_j u
  dxd\xi\\
&& =\sum_{ij=1}^d\iint_{\R^d\times \R^d}\pa_i\pa_j V\{\SI-\SP_0\}
\SX_i u \SY_j \{\SI-\SP_0\}u
  dxd\xi\\
&&=\sum_{ij=1}^d\iint_{\R^d\times \R^d}\pa_i\pa_j
V\SY_j^\ast\{\SI-\SP_0\} \SX_i u  \{\SI-\SP_0\}u
  dxd\xi\\
  &&\leq \eps \|\{\SI-\SP_0\} \SX u\|_\nu^2+\frac{C}{\eps}\iint_{\R^d\times
  \R^d}|\na_x^2 V|^2 | \{\SI-\SP_0\}u|^2 dxd\xi
\end{eqnarray*}
for an arbitrary constant $\eps>0$ to be chosen later, where we used
$\SY \SP_0u=0$ and $\SY^\ast$ is the adjoint operator of $\SY$ given
by
\begin{equation*}
    \SY_j^\ast u=-e^{\frac{|\xi|^2}{4}} \pa_{\xi_j}(e^{-\frac{|\xi|^2}{4}}u
    )=(\frac{1}{2}\xi_j -\pa_{\xi_j})u.
\end{equation*}
Furthermore, due to the assumption \eqref{s3.v.c3} on $V$,
\begin{eqnarray*}
&&\iint_{\R^d\times
  \R^d}|\na_x^2 V|^2 | \{\SI-\SP_0\}u|^2 dxd\xi\\
  &&\leq \de \iint_{\R^d\times
  \R^d}|\na_x V|^2 | \{\SI-\SP_0\}u|^2 dxd\xi +C_\de \iint_{\R^d\times
  \R^d}| \{\SI-\SP_0\}u|^2 dxd\xi\\
 &&=4\de \iint_{\R^d\times
  \R^d}(\frac{1}{4}|\na_x V|^2-\frac{1}{2}\De_x V) | \{\SI-\SP_0\}u|^2 dxd\xi \\
  &&\ \ \ +2\de \iint_{\R^d\times
  \R^d}\De_x V | \{\SI-\SP_0\}u|^2 dxd\xi +C_\de \iint_{\R^d\times
  \R^d}| \{\SI-\SP_0\}u|^2 dxd\xi,
\end{eqnarray*}
which by smallness of $\delta>0$ implies
\begin{equation*}
\iint_{\R^d\times
  \R^d}|\na_x^2 V|^2 | \{\SI-\SP_0\}u|^2 dxd\xi \leq C\de \|\{\SI-\SP_0\}\SX
  u\|^2+ C_\de \|\{\SI-\SP_0\} u\|^2.
\end{equation*}
So, from the above estimates, \eqref{s3.f.p1} is bounded by
\begin{eqnarray*}
\frac{1}{2}\frac{d}{dt}\|\SX u\|^2 &+&\la \|\{\SI-\SP_0\} \SX
u\|_\nu^2 \dis \\
&\leq& \left(\eps+\frac{C\de}{\eps}\right)\| \{\SI-\SP_0\}\SX
u\|_\nu^2 +\frac{C_\de}{\eps}\| \{\SI-\SP_0\} u\|^2.
\end{eqnarray*}
Since $\eps>0$ and $\de>0$ can be arbitrarily small, it follows that
\begin{equation}\label{s3.f.p3}
\frac{1}{2}\frac{d}{dt}\|\SX u\|^2 +\la \|\{\SI-\SP_0\} \SX
u\|_\nu^2\leq C \| \{\SI-\SP_0\} u\|^2.
\end{equation}
Noticing
\begin{equation*}
    \langle \SX_i u, \SY_i u\rangle= \langle  \{\SI-\SP_0\}\SX_i u, \SY_i
    u\rangle,
\end{equation*}
one has from \eqref{s3.f.p2} that
\begin{equation}\label{s3.f.p4}
\frac{1}{2}\frac{d}{dt}\|\SY u\|^2 +\la \|\{\SI-\SP_0\} \SY
u\|_\nu^2\leq C \| \{\SI-\SP_0\} \SX u\|^2.
\end{equation}
Finally, we turn to the estimate on $\SP_0 \SX u$ and $\SP_0 u$.
Similar to obtain \eqref{s2.f.p3}, \eqref{s3.f.e}$_1$ also gives the
same fluid-type system:
\begin{equation*}
    \left\{\begin{array}{l}
     \dis \pa_t a -\SX^\ast\cdot b=0,\\[3mm]
     \dis \pa_t b+\SX a -\SX^\ast\cdot  \Ga(\{\SI-\SP_0\}u)+b=0,
    \end{array}\right.
\end{equation*}
with
\begin{equation*}
    \left.\int_{\R^d} e^{-\frac{V(x)}{2}}a(t,x) dx\right|_{t\geq
    0}=0.
\end{equation*}
again due to the mass conservation and initial condition
\eqref{s3.v.c4}. Applying the same argument as before and then using
Poincar\'{e} inequality in Proposition \ref{prop.pn} by assumptions
\eqref{s3.v.c1}-\eqref{s3.v.c2}, one has
\begin{equation}\label{s3.f.p5}
    \frac{d}{dt} (\SX a, b)+\la   (\|\SX a\|^2+\|a\|^2) \leq C(\|\{\SI-\SP_0\}\SX u\|^2+\|\{\SI-\SP_0\}u\|^2).
\end{equation}

Now, let us define a temporal functional
\begin{equation*}
    \CE(u(t))=\|u\|^2+\kappa_1\|\SX u\|^2+\kappa_2\|\SY u\|^2+ \kappa_3 (\SX a,
    b)
\end{equation*}
with constants $\kappa_i>0$ $(1\leq i\leq 3)$. One can choose
\begin{equation*}
    0<\kappa_3\ll \kappa_2\ll\kappa_1\ll 1
\end{equation*}
such that
\begin{equation*}
    \CE(u(t))\sim \|u(t)\|_{\CH^1}^2
\end{equation*}
and further the linear combination of \eqref{s3.f.p0},
\eqref{s3.f.p3}, \eqref{s3.f.p4} and \eqref{s3.f.p5} gives
\begin{equation*}
    \frac{d}{dt}\CE(u(t))+\la \CE(u(t))\leq 0.
\end{equation*}
By Gronwall's inequality, this proves \eqref{thm.f.1} and hence
completes the proof of Theorem \ref{thm.f} for Model 2.

\section{Boltzmann equation}

In this section we prove Theorem \ref{thm.nf} and Theorem
\ref{thm.f} for Model 3. Although this can be done along the same
line as for Model 1 and Model 2, it is a little more complicated for
Model 3 since ${\rm ker}\,\SL=d+2$ so that the linearized Boltzmann
operator $\SL$ is degenerate over a space with higher dimensions.
The key idea is still based on the macro-micro decomposition and
Kawashima's dissipation argument on the hyperbolic-parabolic system.
The additional difficulty in the presence of confining forces lies
in verifying the Korn-type inequality to obtain the dissipation of
momentum component in the fluid part, which is left to the next
section.

\bigskip

\noindent{\it 4.1} {\it Case when $\Omega=\R^d$ and $V=0$}. We
consider the Cauchy problem
\begin{equation}\label{s4.e}
\left\{\begin{array}{rll} \dis \pa_t u+\xi \cdot \na_x u &= \SL u+
h, &\ \ t> 0,x\in
  \R^d,\xi\in \R^d,\\[2mm]
 \dis  u(0,x,\xi)&=u_0(x,\xi), &\ \ x\in
  \R^d,\xi\in \R^d,
\end{array}\right.
\end{equation}
where $\SL$ is the linearized Boltzmann collision operator given in
\eqref{def.lm3}, and as in \eqref{thm.nf.ch}, $h=h(t,x,\xi)$
satisfies
\begin{equation*}
    h(t,x,\xi)\perp {\rm ker}\,\SL,\ \ \forall\, t\geq 0,x\in \R^d.
\end{equation*}
As before, the solution to \eqref{s4.e} can be written as
\begin{equation*}
    u(t)=e^{t\SB}u_0 +\int_0^t e^{-(t-s)\SB} h(s)ds,
\end{equation*}
where $\SB=-\xi\cdot \na_x+\SL$. Both parts in the above $u(t)$ will
be estimated in a unifying way by using the method of Fourier
transform similarly before.

Firstly, from \eqref{s4.e}$_1$, it is straightforward to obtain
\begin{equation}\label{s4.nf.hp1}
    \frac{1}{2} \pa_t \|\hat{u}\|_{L^2_\xi}^2 +\la
    \iint_{\R^d\times\R^d} \nu(\xi)|\{\SI-\SP\} \hat{u}|^2dxd\xi \leq
    C\|\nu^{-1/2}\hat{h}\|_{L^2_\xi}^2,
\end{equation}
where $\nu(\xi)\sim 1+|\xi|$ defined by \eqref{def.nu.m3} is the
collision frequency for the case of hard sphere model \cite{Gl}, and
the orthogonal velocity projection operator $\SP: L^2_\xi \to {\rm
ker}\,\SL$ is described by \eqref{def.pm3}. Next, we devote
ourselves to the estimate on $\SP u$ or equivalently $(a,b,c)$. In
fact, by taking $\frac{1}{2}d(d+5)+2$ number of velocity moments
\begin{equation*}
    M^{1/2}, \xi_iM^{1/2}, (|\xi|^2-d)M^{1/2},
    (\xi_i\xi_j-1)M^{1/2}, (|\xi|^2-d-2)\xi_i M^{1/2}
\end{equation*}
with $1\leq i,j\leq d$ for the equation \eqref{s4.e}$_1$, one has
the fluid-type system
\begin{equation}\label{s4.nf.fs}
    \left\{\begin{array}{l}
      \dis \pa_t a +\na_x\cdot b =0,\\[3mm]
      \dis \pa_t b +\na_x (a+2c)+\na_x\cdot \Ga (\{\SI-\SP\} u)=0,\\[3mm]
      \dis \pa_t c +\frac{1}{d}\na_x\cdot b +\frac{1}{2d}\na_x\cdot
      \Lambda (\{\SI-\SP\} u)=0,\\[3mm]
      \dis \pa_t [\Ga_{ij}(\{\SI-\SP\} u)+2c\de_{ij}]+\pa_ib_j+\pa_j
      b_i=\Ga_{ij}(r+h),\\[3mm]
      \dis \pa_t \Lambda_i(\{\SI-\SP\} u)+\pa_i c = \Lambda_i(r+h),
    \end{array}\right.
\end{equation}
where $a,b,c$ are defined in \eqref{def.pm3}, the matrix-valued
function $\Ga$ is defined by \eqref{def.Ga}, the moment function
$\Lambda=(\Lambda_i)_{1\leq i\leq d}$ is defined by
\begin{equation}\label{def.Lam}
\Lambda_i(g)=\langle(|\xi|^2-d-2)\xi_i M^{1/2}, g\rangle,
\end{equation}
and $r$ is denoted by
\begin{equation*}
    r=-\xi\cdot \na_x\{\SI-\SP\} u +\SL \{\SI-\SP\} u.
\end{equation*}
The detailed derivation of \eqref{s4.nf.fs} was given in
\cite{Duan-Torus} and thus is omitted for simplicity.

\begin{lemma}\label{lem.m3f}
There exists a functional $E_{\rm int}(\hat{u})$ given by
\begin{eqnarray}
 E_{\rm int}(\hat{u})  &=& \frac{1}{1+|k|^2}\sum_{i=1}^d (\rmi k_i \hat{c}\mid
 \Lambda_i
 (\{\SI-\SP\}\hat{u}))\nonumber \\
 &&+\frac{\kappa_1}{1+|k|^2}\sum_{ij=1}^d (\rmi k_i \hat{b}_j+\rmi k_j \hat{b}_i\mid
 \Ga_{ij}(\{\SI-\SP\}\hat{u})+2\hat{c}\de_{ij})\nonumber \\
 &&+\frac{\kappa_2}{1+|k|^2}\sum_{i=1}^d (\rmi k_i \hat{a}\mid \hat{b}_i)\label{lem.m3f.1}
\end{eqnarray}
for constants $0<\kappa_2\ll\kappa_1\ll 1$ such that there are
constants $\la>0$, $C$ such that
\begin{equation}
\pa_t \rmre E_{\rm
int}(\hat{u})+\frac{\la|k|^2}{1+|k|^2}(|\hat{a}|^2+|\hat{b}|^2+|\hat{c}|^2)\leq
C(\|\{\SI-\SP\}\hat{u}\|_{L^2_\xi}^2+\|\nu^{-1/2}\hat{h}\|_{L^2_\xi}^2)\label{lem.m3f.2}
\end{equation}
holds for any $t\geq 0$ and $k\in \R^d$.
\end{lemma}

\begin{proof}
Write \eqref{s4.nf.fs} in terms of Fourier transform as
\begin{equation}\label{lem.m3f.p1}
    \left\{\begin{array}{l}
      \dis \pa_t \hat{a} +\rmi k\cdot \hat{b} =0,\\[3mm]
      \dis \pa_t \hat{b} +\rmi (\hat{a}+2\hat{c})+\rmi \Ga (\{\SI-\SP\} \hat{u})\cdot k=0,\\[3mm]
      \dis \pa_t \hat{c} +\frac{1}{d}\rmi k\cdot \hat{b} +\frac{1}{2d}\rmi k\cdot
      \Lambda (\{\SI-\SP\} \hat{u})=0,\\[3mm]
      \dis \pa_t [\Ga_{ij}(\{\SI-\SP\} \hat{u})+2\hat{c}\de_{ij}]+\rmi
      k_i\hat{b}_j+\rmi k_j
      \hat{b}_i=\Ga_{ij}(\hat{r}+\hat{h}),\\[3mm]
      \dis \pa_t \Lambda_i(\{\SI-\SP\} \hat{u})+\rmi k_i \hat{c} = \Lambda_i(\hat{r}+\hat{h}),
    \end{array}\right.
\end{equation}
for $1\leq i,j\leq d$, where
\begin{equation*}
    \hat{r}=-\rmi \xi\cdot k \{\SI-\SP\} \hat{u} +\SL \{\SI-\SP\} \hat{u}.
\end{equation*}

\medskip

\noindent{\it Step 1.} Estimate $\hat{c}$: Starting from the
highest-order moment equation \eqref{lem.m3f.p1}$_5$, one has
\begin{eqnarray*}
  |k|^2|\hat{c}|^2 &=& \sum_{i=1}^d (\rmi k_i \hat{c}\mid \rmi k_i
  \hat{c})=\sum_{i=1}^d (\rmi k_i \hat{c}\mid  -\pa_t \Lambda_i(\{\SI-\SP\} \hat{u})+\Lambda_i(\hat{r}+\hat{h}))\\
  &=&-\pa_t \sum_{i=1}^d (\rmi k_i \hat{c}\mid  \Lambda_i(\{\SI-\SP\} \hat{u}))
  + \sum_{i=1}^d (\rmi k_i \pa_t\hat{c}\mid  \Lambda_i(\{\SI-\SP\}
  \hat{u}))\\
  &&
  + \sum_{i=1}^d (\rmi k_i \hat{c}\mid\Lambda_i(\hat{r}+\hat{h}) ).
\end{eqnarray*}
Using \eqref{lem.m3f.p1}$_3$ to replace $\pa_t\hat{c}$ gives
\begin{eqnarray*}
&&\sum_{i=1}^d (\rmi k_i \pa_t\hat{c}\mid  \Lambda_i(\{\SI-\SP\}
  \hat{u}))\\
  &&= \sum_{i=1}^d (\frac{1}{d}\rmi k\cdot \hat{b} +\frac{1}{2d}\rmi k\cdot
      \Lambda (\{\SI-\SP\} \hat{u}) \mid  \rmi k_i\Lambda_i(\{\SI-\SP\}
  \hat{u}))\\
  &&\leq \eps_1 |k\cdot \hat{b}|^2+\frac{C}{\eps_1}\|\{\SI-\SP\}
  \hat{u}\|_{L^2_\xi}^2,
\end{eqnarray*}
where $0<\eps_1\leq 1$ is arbitrary to be chosen later. Moreover, it
holds that
\begin{equation*}
\sum_{i=1}^d (\rmi k_i \hat{c}\mid\Lambda_i(\hat{r}+\hat{h}) )\leq
\frac{1}{2}|k|^2|\hat{c}|^2 +C
(1+|k|^2)\|\{\SI-\SP\}\hat{u}\|_{L^2_\xi}^2
+C\|\nu^{-1/2}\hat{h}\|_{L^2_\xi}^2.
\end{equation*}
Then, it follows
\begin{eqnarray}
&&\pa_t \rmre\,\sum_{i=1}^d (\rmi k_i \hat{c}\mid
\Lambda_i(\{\SI-\SP\} \hat{u}))+\la  |k|^2|\hat{c}|^2\nonumber \\
&&\hspace{2cm}\leq \eps_1 |k\cdot \hat{b}|^2+\frac{C}{\eps_1}
(1+|k|^2)\|\{\SI-\SP\}\hat{u}\|_{L^2_\xi}^2
+C\|\nu^{-1/2}\hat{h}\|_{L^2_\xi}^2.\label{lem.m3f.p2}
\end{eqnarray}

\medskip

\noindent{\it Step 2.} Estimate $\hat{b}$:
 Observe the identity
\begin{equation*}
    \sum_{ij=1}^d\|\rmi k_i \hat{b}_j+\rmi k_j
    \hat{b}_i\|^2=2|k|^2|\hat{b}|^2+2\|k\cdot \hat{b}\|^2.
\end{equation*}
On the other hand, compute from  \eqref{lem.m3f.p1}$_4$ that
\begin{eqnarray*}
&&\sum_{ij=1}^d\|\rmi k_i \hat{b}_j+\rmi k_j
    \hat{b}_i\|^2\\
&&=\sum_{ij=1}^d ( \rmi k_i \hat{b}_j+\rmi k_j
    \hat{b}_i \mid  -\pa_t [\Ga_{ij}(\{\SI-\SP\}
    \hat{u})+2\hat{c}\de_{ij}]+\Ga_{ij}(\hat{r}+\hat{h}))\\
&&=-\pa_t \sum_{ij=1}^d ( \rmi k_i \hat{b}_j+\rmi k_j
    \hat{b}_i \mid \Ga_{ij}(\{\SI-\SP\}
    \hat{u})+2\hat{c}\de_{ij})\\
&&\ \ \ + \sum_{ij=1}^d ( \rmi k_i \pa_t\hat{b}_j+\rmi k_j
    \pa_t\hat{b}_i \mid \Ga_{ij}(\{\SI-\SP\}
    \hat{u})+2\hat{c}\de_{ij}) \\
&&\ \ \ + \sum_{ij=1}^d ( \rmi k_i \hat{b}_j+\rmi k_j
    \hat{b}_i \mid \Ga_{ij}(\hat{r}+\hat{h})).
\end{eqnarray*}
Using \eqref{lem.m3f.p1}$_2$ to replace $\pa_t\hat{b}$ gives
\begin{eqnarray*}
&&\sum_{ij=1}^d ( \rmi k_i \pa_t\hat{b}_j+\rmi k_j
    \pa_t\hat{b}_i \mid \Ga_{ij}(\{\SI-\SP\}
    \hat{u})+2\hat{c}\de_{ij})\\
&&=2\sum_{ij=1}^d (\pa_t\hat{b}_i \mid  \rmi k_j\Ga_{ij}(\{\SI-\SP\}
    \hat{u})+2 \rmi k_j\hat{c}\de_{ij})\\
&&=2\sum_{ij=1}^d (\rmi k_i (\hat{a}+2\hat{c})+\rmi \sum_{\ell=1}^d
k_\ell\Ga_{\ell i} (\{\SI-\SP\} \hat{u})\mid  \rmi
k_j\Ga_{ij}(\{\SI-\SP\}
    \hat{u})+2 \rmi k_j\hat{c}\de_{ij})\\
&&\leq \eps_2 |k|^2 |\hat{a}|^2 +\frac{C}{\eps_2}|k|^2|\hat{c}|^2
+\frac{C}{\eps_2}|k|^2 \|\{\SI-\SP\} \hat{u}\|_{L^2_\xi}^2,
\end{eqnarray*}
where we used the symmetry of $\Ga_{ij}$ and $0<\eps_2\leq 1$ is
arbitrary to be chosen later. Moreover, it is straightforward to
obtain
\begin{eqnarray*}
&\dis  \sum_{ij=1}^d ( \rmi k_i \hat{b}_j+\rmi k_j
    \hat{b}_i \mid \Ga_{ij}(\hat{r}+\hat{h}))\\
&\dis    \leq \frac{1}{2} \sum_{ij=1}^d\|\rmi k_i \hat{b}_j+\rmi k_j
    \hat{b}_i\|^2 +C (1+|k|^2) \|\{\SI-\SP\} \hat{u}\|_{L^2_\xi}^2
    +C\|\nu^{-1/2}\hat{h}\|_{L^2_\xi}^2.
\end{eqnarray*}
So, collecting the above estimates, one has
\begin{eqnarray}
&&\pa_t \rmre\,  \sum_{ij=1}^d ( \rmi k_i \hat{b}_j+\rmi k_j
    \hat{b}_i \mid \Ga_{ij}(\{\SI-\SP\}
    \hat{u})+2\hat{c}\de_{ij}) +|k|^2|\hat{b}|^2 \nonumber \\
&&\hspace{1cm}    \leq \eps_2 |k|^2
|\hat{a}|^2+\frac{C}{\eps_2}|k|^2|\hat{c}|^2 +\frac{C}{\eps_2}|k|^2
\|\{\SI-\SP\}
\hat{u}\|_{L^2_\xi}^2+C\|\nu^{-1/2}\hat{h}\|_{L^2_\xi}^2.\label{lem.m3f.p3}
\end{eqnarray}

\medskip

\noindent{\it Step 3.} Estimate $\hat{a}$: Similarly,
\eqref{lem.m3f.p1}$_2$ implies
\begin{eqnarray*}
  |k|^2|\hat{a}|^2 &=& \sum_{i=1}^d (\rmi k_i \hat{a}\mid \rmi k_i \hat{a})
  = \sum_{i=1}^d (\rmi k_i \hat{a}\mid  -\pa_t \hat{b}_i -2\rmi k_i\hat{c}-\sum_{j=1}^d\rmi k_j
  \Ga_{ij} (\{\SI-\SP\} \hat{u}))\\
  &=&-\pa_t  \sum_{i=1}^d (\rmi k_i \hat{a}\mid \hat{b}_i)+\sum_{i=1}^d (\rmi k_i \pa_t\hat{a}\mid
  \hat{b}_i)\\
  &&+\sum_{i=1}^d (\rmi k_i \hat{a}\mid -2\rmi k_i\hat{c}-\sum_{j=1}^d\rmi k_j
  \Ga_{ij} (\{\SI-\SP\} \hat{u})
\end{eqnarray*}
which after using \eqref{lem.m3f.p1}$_1$ to replace $\pa_t\hat{a}$
on the r.h.s. gives
\begin{eqnarray}
 \pa_t \rmre\,  \sum_{i=1}^d (\rmi k_i \hat{a}\mid \hat{b}_i)+\la
 |k|^2|\hat{a}|^2 \leq |k\cdot \hat{b}|^2 +C|k|^2 |\hat{c}|^2
 +C|k|^2\|\{\SI-\SP\} \hat{u}\|_{L^2_\xi}^2.\label{lem.m3f.p4}
\end{eqnarray}

\medskip

Finally, define $E_{\rm int}(\hat{u})$ by \eqref{lem.m3f.1}. Then,
for properly chosen constants $0<\kappa_2\ll \kappa_1\ll 1$ and
small constants $\eps_1>0,\eps_2>0$, \eqref{lem.m3f.2} follows from
the linear combination of the above three estimates
\eqref{lem.m3f.p2}, \eqref{lem.m3f.p3} and \eqref{lem.m3f.p4} and
further dividing it by $1+|k|^2$. This completes the proof of Lemma
\ref{lem.m3f}.
\end{proof}

\begin{remark}
As in \eqref{s2.w.p2}, \eqref{s2.f.p5} and \eqref{s3.f.p5}, Lemma
\ref{lem.m3f} shows that the dissipation of degenerate part which is
the kernel of the linearized Boltzmann equation with dimensions
equal to $d+2$ can be recovered from the fluid-type moment system
\eqref{s4.nf.fs}. This property was firstly observed by Guo
\cite{Guo-CMP,Guo-IUMJ}, and it was later improved in
\cite{Duan,Duan-Torus,DS09} and \cite{Guo-diff,Jang} for different
purposes. Precisely, \cite{Guo-IUMJ} is mainly based on an
elliptic-type equation of $b$ derived from \eqref{s4.nf.fs}$_3$-
\eqref{s4.nf.fs}$_4$. The aim of \cite{Duan} is to remove time
derivatives by constructing some functional which is similar to
\eqref{lem.m3f.1} but takes more complicated form. \cite{Duan-Torus}
introduced moment functions $\Ga$ and $\Lambda$ to refine the form
of \eqref{lem.m3f.1}. \cite{DS09} exactly used the same method to
deal with the Vlasov-Poisson-Boltzmann system. Here, it should be
emphasized that due to new estimates on the dissipation of $b$, the
current method of proof is more general in the sense that it can be
directly modified to apply to the case with a potential forcing, see
Lemma \ref{lem.s4.fd}. Notice that three terms in \eqref{lem.m3f.1}
are inner products of $i$-th order moment and $(i+1)$-th order
moment for $i=0,1,2$, respectively, and also it is in the same
spirit of Kawashima's construction of compensation functions
\cite{Ka-13,Gl}.
\end{remark}

Now, for $t\geq 0$, $k\in \R^d$, define
\begin{equation*}
    E(\hat{u})=\|\hat{u}\|_{L^2_\xi}^2+ \kappa_3 \rmre E_{\rm int}(\hat{u})
\end{equation*}
where $E_{\rm int}(\hat{u})$ is defined by \eqref{lem.m3f.1} and
$\kappa_3>0$ is to be chosen. Similarly before, one can let
$\kappa_3>0$ be small such that $E(\hat{u})\sim
\|\hat{u}\|_{L^2_\xi}^2 $ and  the linear combination of
\eqref{s4.nf.hp1} and \eqref{lem.m3f.2} gives
\begin{equation*}
    \pa_t E(\hat{u}) + \frac{\la |k|^2}{1+|k|^2} E(\hat{u})\leq
    C\|\nu^{-1/2}\hat{h}\|_{L^2_\xi}^2.
\end{equation*}
Hence, \eqref{thm.nf.1} and \eqref{thm.nf.2} follows from the above
estimate in the same way as before. This completes the proof of
\eqref{thm.nf.1} and \eqref{thm.nf.2} for Model 3 in Theorem
\ref{thm.nf}.

\bigskip

\noindent{\it 4.2} {\it Case when $\Omega=\T^d$ and $V=0$}. In this
case, \eqref{thm.nf.3} follows from the same argument as for Model 1
and Model 2. In fact, it suffices to check
\begin{equation*}
    \left.\int_{\T^d} (a,b,c)dx\right|_{t\geq 0}\equiv 0
\end{equation*}
which results from the fact that it initially holds at $t=0$ by
\eqref{thm.nf.cid} and the conservation laws
\begin{equation*}
    \frac{d}{dt}\int_{\T^d} (a,b,c)dx= 0
\end{equation*}
holds by \eqref{s4.nf.fs}$_1$-\eqref{s4.nf.fs}$_3$. This completes
the proof of \eqref{thm.nf.3} in Theorem \ref{thm.nf} for Model 3.

\bigskip

\noindent{\it 4.3} {\it Case when $\Omega=\R^d$ and $V$ is
confining}. In particular, let $V=\frac{|x|^2}{2}-\frac{d}{2}\ln
(2\pi)$ and $d\geq 3$. Take $u_0$ with
\begin{equation}\label{s4.f3.id}
    \iint_{\R^d\times\R^d}(1,x,\xi,x\cdot \xi,x\times
    \xi,|x|^2,|\xi|^2) \CM^{1/2}u_0 dxd\xi=0,
\end{equation}
where $(x\times \xi)_{ij}=x_i\xi_j-x_j\xi_i$ for $1\leq i,j\leq d$.
Let $u(t)=e^{t\SB}u_0$ be the solution to the Cauchy problem
\begin{equation}\label{s4.f3.e}
\left\{\begin{array}{rll}
 \dis \pa_t u+\ST u &= \SL u, &\ \ t> 0,x\in
  \R^d,\xi\in \R^d,\\[2mm]
  \dis u(0,x,\xi)&=u_0(x,\xi),&\ \ x\in
  \R^d,\xi\in \R^d.
\end{array}\right.
\end{equation}
Firstly, we verify that the property \eqref{s4.f3.id} can be
preserved at all time $t\geq 0$ for  the above Cauchy problem
\eqref{s4.f3.e}.

\begin{lemma}\label{lem.s4.con}
Under the assumption \eqref{s4.f3.id},
\begin{equation}\label{lem.s4.con.1}
    \iint_{\R^d\times\R^d}(1,x,\xi,x\cdot \xi,x\times
    \xi,|x|^2,|\xi|^2)\CM^{1/2}u(t,x,\xi) dxd\xi=0
\end{equation}
holds for any $t\geq 0$. Particularly, for any $t\geq 0$,
\begin{eqnarray*}
&\dis \int_{\R^d} e^{-\frac{V}{2}}a(t,x)dx=\int_{\R^d}
e^{-\frac{V}{2}}b_i(t,x)dx=\int_{\R^d} e^{-\frac{V}{2}}c(t,x)dx=0,\
1\leq i\leq d,\\
&\dis  \int_{\R^d} e^{-\frac{V}{2}} x\times b(t,x) dx=0.
\end{eqnarray*}
\end{lemma}

\begin{proof}
Let $(\!(\cdot,\cdot)\!)$ denote the usual inner product in
$L^2_{x,\xi}$.  For simplicity, define also the weighted inner
product $(\!(\cdot,\cdot)\!)_{\CM}$ by
\begin{equation*}
(\!(u,v)\!)_{\CM}=\iint_{\R^d\times\R^d} u v \CM^{1/2}dxd\xi.
\end{equation*}
Notice that $f=\CM+\CM^{1/2}u$ satisfies \eqref{def.lm3-1}. From
integration by parts, it is easy to obtain the following ODE system
\begin{equation*}
\left\{\begin{array}{l} \dis  \frac{d}{dt}(\!(1,f)\!)=0,\ \
    \frac{d}{dt}(\!(x\times \xi,f)\!)=0,\\[3mm]
\dis \frac{d}{dt}(\!(x,f)\!)=(\!(\xi,f)\!),\ \
\frac{d}{dt}(\!(\xi,f)\!)=-(\!(x,f)\!),\\[3mm]
\dis  \frac{d}{dt}(\!( x\cdot \xi,f)\!)=(\!(|\xi|^2-|x|^2,f)\!),\\[3mm]
\dis  \frac{d}{dt}(\!(|x|^2,f)\!)=2(\!(x\cdot \xi,f)\!),\ \ \dis
\frac{d}{dt}(\!(|\xi|^2,f)\!)=-2(\!(x\cdot \xi,f)\!).
\end{array}\right.
\end{equation*}
Equivalently, in terms of $u$, it holds that
\begin{equation}\label{lem.s4.con.p1}
\left\{\begin{array}{l} \dis \frac{d}{dt}(\!(1,u)\!)_{\CM}=0,\ \
    \frac{d}{dt}(\!(x\times \xi,u)\!)_{\CM}=0,\\[3mm]
\dis \frac{d}{dt}(\!(x,u)\!)_{\CM}=(\!(\xi,u)\!)_{\CM},\ \
\frac{d}{dt}(\!(\xi,u)\!)_{\CM}=-(\!(x,u)\!)_{\CM},\\[3mm]
\dis  \frac{d}{dt}(\!( x\cdot \xi,u)\!)_{\CM}=(\!(|\xi|^2-|x|^2,u)\!)_{\CM},\\[3mm]
\dis  \frac{d}{dt}(\!(|x|^2,u)\!)_{\CM}=2(\!(x\cdot
\xi,u)\!)_{\CM},\ \ \dis
\frac{d}{dt}(\!(|\xi|^2,u)\!)_{\CM}=-2(\!(x\cdot \xi,u)\!)_{\CM},
\end{array}\right.
\end{equation}
where we used the fact
\begin{equation*}
(\!(|\xi|^2-|x|^2,\CM)\!)=0,\ \
\CM=(2\pi)^{-d}e^{-\frac{|x|^2+|\xi|^2}{2}}.
\end{equation*}
Define temporal functions
\begin{equation*}
    y_i(t)=(\!(\xi_i,u)\!)_{\CM}\ \ (1\leq i\leq d),\ \ \ z(t)=(\!(x\cdot
    \xi,u)\!)_{\CM}.
\end{equation*}
Then, from \eqref{lem.s4.con.p1} and \eqref{lem.s4.con.1}, $y_i(t)$
and $z(t)$ satisfy the initial value problems of the linear
second-order ODE as follows
\begin{equation*}
   \left\{\begin{array}{l}
    \dis y_i''(t)+y_i(t)=0, \ \ t> 0\\[3mm]
   \dis  y_i(0)=(\!(\xi,u_0)\!)_{\CM}=0,\\[3mm]
   \dis y_i'(0)=-(\!(x_i,u_0)\!)_{\CM}=0,
    \end{array}\right.
\end{equation*}
and
\begin{equation*}
   \left\{\begin{array}{l}
    \dis z''(t)+ 4z(t)=0, \ \ t> 0\\[3mm]
   \dis  z(0)=(\!(x\cdot \xi,u_0)\!)_{\CM}=0,\\[3mm]
   \dis z'(0)=-(\!(|\xi|^2-|x|^2,u_0)\!)_{\CM}=0.
    \end{array}\right.
\end{equation*}
Hence, both solutions are trivial, i.e.
\begin{equation*}
    y_i(t)=0 \ (1\leq i\leq d)\ \ \text{and}\ \ z(t)=0
\end{equation*}
for any $t\geq 0$. Putting this back to the system
\eqref{lem.s4.con.p1} implies that all inner products in
\eqref{lem.s4.con.p1} vanish due to \eqref{lem.s4.con.1}. This
completes the proof of Lemma \ref{lem.s4.con}.
\end{proof}

Next, we proceed along the same line of proof as for Model 1. The
zero-order energy integration of \eqref{s4.f3.e}$_1$ gives
\begin{equation}\label{s4.f3.p1}
    \frac{1}{2}\frac{d}{dt}\|u\|^2+\la \iint_{\R^d\times
    \R^d}\nu(\xi)|\{\SI-\SP\}u|^2dxd\xi\leq 0.
\end{equation}
Similar to obtain \eqref{s2.f.p0-1}, for $\SX$ and $\SY$
differentiations, one has
\begin{eqnarray*}
&& \frac{1}{2}\frac{d}{dt}(\|\SX u\|^2+\|\SY u\|^2)+\la
\iint_{\R^d\times
    \R^d}\nu(\xi)|\{\SI-\SP\}\SX u|^2dxd\xi\\
&&\hspace{5mm}+\la \iint_{\R^d\times
    \R^d}\nu(\xi)|\{\SI-\SP\}\SY u|^2dxd\xi\leq -\sum_{i=1}^d
    \iint_{\R^d\times
    \R^d}[\SL,\SY_i]u \SY_i u dxd\xi.
\end{eqnarray*}
It is easy to check $\SY_i \SP u \in {\rm ker}\,\SL$ which implies
\begin{equation*}
[\SL,\SY_i]u=\SL \SY_i u-\SY_i \SL u=\SL \SY_i \{\SI-\SP\}u-\SY_i
\SL \{\SI-\SP\}u=[\SL,\SY_i] \{\SI-\SP\}u.
\end{equation*}
Since $[\SL,\SY_i]$ is a bounded operator from $L^2_\xi$ to
$L^2_\xi$ by \cite{DUYZ},
\begin{equation*}
 -\sum_{i=1}^d
    \iint_{\R^d\times
    \R^d}[\SL,\SY_i]u \SY_i u dxd\xi\leq \eps
    \|(a,b,c)\|^2+\frac{C}{\eps}\|\{\SI-\SP\}u\|^2
\end{equation*}
for an arbitrary constant $0<\eps\leq 1$ to be chosen later. Then,
it follows
\begin{eqnarray}
&& \frac{1}{2}\frac{d}{dt}(\|\SX u\|^2+\|\SY u\|^2)+\la
\iint_{\R^d\times
    \R^d}\nu(\xi)|\{\SI-\SP\}\SX u|^2dxd\xi \nonumber \\
&&\hspace{5mm}+\la \iint_{\R^d\times
    \R^d}\nu(\xi)|\{\SI-\SP\}\SY u|^2dxd\xi\leq  \eps
    \|(a,b,c)\|^2+\frac{C}{\eps}\|\{\SI-\SP\}u\|^2.\label{s4.f3.p2}
\end{eqnarray}

The rest is to deal with the dissipation rate corresponding to $\SP
\SX u$ and $\SP u$ or equivalently $\SX (a,b,c)$ and $(a,b,c)$.
Poincar\'{e} and Korn-type inequalities given in the next section
play a key role in this step. In fact, similar to get
\eqref{s4.nf.fs}, from the equation \eqref{s4.f3.e}$_1$ in the
presence of potential forcing, one can obtain the following
fluid-type system:
\begin{equation}\label{s4.f.fs}
    \left\{\begin{array}{l}
      \dis \pa_t a -\SX^\ast\cdot b =0,\\[3mm]
      \dis \pa_t b +\SX (a+2c)-\SX^\ast\cdot \Ga (\{\SI-\SP\} u)=0,\\[3mm]
      \dis \pa_t c +\frac{1}{d}\SX\cdot b -\frac{1}{2d}\SX^\ast\cdot
      \Lambda (\{\SI-\SP\} u)=0,\\[3mm]
      \dis \pa_t [\Ga_{ij}(\{\SI-\SP\} u)+2c\de_{ij}]+\SX_ib_j+\SX_j
      b_i=\Ga_{ij}(r),\\[3mm]
      \dis \pa_t \Lambda_i(\{\SI-\SP\} u)+\SX_i c = \Lambda_i(r),
    \end{array}\right.
\end{equation}
where moment functions $\Ga$ and $\Lambda$ are defined as in
\eqref{def.Ga} and \eqref{def.Lam} respectively, and for simplicity,
we still used $r$ to denote
\begin{equation*}
    r=-\xi\cdot \na_x\{\SI-\SP\}u +\na_x V\cdot \na_\xi \{\SI-\SP\}u
    +\SL \{\SI-\SP\}u.
\end{equation*}

\begin{lemma}\label{lem.s4.fd}
There exists a temporal functional $\CE_{\rm int} (u(t))$ given by
\begin{eqnarray}
\CE_{\rm int} (u(t))&=& \sum_{i=1}^d (\SX_i c,\Lambda_i(\{\SI-\SP\}u)) \nonumber \\
&&+\kappa_1 \sum_{ij=1}^d (\SX_i b_j+\SX_j b_i,\Ga_{ij}(\{\SI-\SP\}u)+2c\delta_{ij}) \nonumber \\
&&+\kappa_2 \sum_{i=1}^d (\SX_i a, b_i)\label{lem.s4.fd.1}
\end{eqnarray}
for constants $0<\kappa_2\ll \kappa_1\ll 1$ such that there are
constants $\la>0$, $C$ such that
\begin{equation}\label{lem.s4.fd.2}
    \frac{d}{dt}\CE_{\rm int} (u(t))+\la \CD(a,b,c)\leq C (\|\{\SI-\SP\}\SX u\|^2+\|\{\SI-\SP\}u\|^2)
\end{equation}
holds for any $t\geq 0$, where the dissipation rate $\CD(a,b,c)$
denotes
\begin{eqnarray*}
 \CD(a,b,c) &=& (\|\SX a\|^2+\|a\|^2) + \sum_{i=1}^d (\|\SX
 b_i\|^2+\|b_i\|^2)+ (\|\SX c\|^2+\|c\|^2).
\end{eqnarray*}
\end{lemma}

\begin{proof}
This follows from the same procedure as for the proof of Lemma
\ref{lem.m3f} which is on the basis of Fourier transform when there
is no external forcing. For completeness, we shall provide all
details of proof. Firstly, it is straightforward to check
\begin{eqnarray}
  \|\Ga(r)\|^2+\|\Lambda (r)\|^2 &\leq & C\iint_{\R^d\times\R^d}
  |\na_x\{\SI-\SP\}u|^2 +|\{\SI-\SP\}u|^2 (1+|\na_x V|^2)dx d\xi \nonumber \\
  &\leq & C (\|\{\SI-\SP\}\SX u\|^2+\|\{\SI-\SP\}u\|^2).\label{lem.s4.fd.p1}
\end{eqnarray}

\medskip

\noindent{\it Step 1.} Estimate on $c$: \eqref{s4.f.fs}$_5$ gives
\begin{eqnarray}
  \|\SX c\|^2 &=& \sum_{i=1}^d (\SX_i c,-\pa_t\Lambda_i
  (\{\SI-\SP\}u)+\Lambda_i(r)) \nonumber \\
  &=&-\frac{d}{dt}\sum_{i=1}^d (\SX_i c, \Lambda_i(\{\SI-\SP\}u))
  +\sum_{i=1}^d(\SX_i\pa_t c, \Lambda_i(\{\SI-\SP\}u)) \nonumber \\
  &&
  +\sum_{i=1}^d(\SX_i c, \Lambda_i(r)).\label{lem.s4.fd.p2}
\end{eqnarray}
Using  \eqref{s4.f.fs}$_3$ to replace $\pa_t c$, one has
\begin{eqnarray*}
&&\sum_{i=1}^d(\SX_i\pa_t c, \Lambda_i(\{\SI-\SP\}u))\\
&&=(-\frac{1}{d}\SX\cdot b +\frac{1}{2d}\SX^\ast\cdot \Lambda
(\{\SI-\SP\}u), \SX^\ast\cdot \Lambda(\{\SI-\SP\}u))\\
&&\leq \eps_1\|\SX\cdot b\|^2 +\frac{C}{\eps_1}\|\SX^\ast
\{\SI-\SP\}u\|^2
\end{eqnarray*}
for an arbitrary constant $0<\eps_1\leq 1$ to be chosen later. In
addition, it holds
\begin{equation*}
\sum_{i=1}^d(\SX_i c, \Lambda_i(r))\leq \frac{1}{2}\|\SX c\|^2
+C\|\Lambda (r)\|^2.
\end{equation*}
By putting the above two estimates into \eqref{lem.s4.fd.p2} and
using  \eqref{lem.s4.fd.p1} and Lemma \ref{lem.ineq.V}, one has
\begin{eqnarray}
&&\frac{d}{dt}\sum_{i=1}^d (\SX_i c,\Lambda_i(\{\SI-\SP\}u))+\la
(\|\SX c\|^2+\|c\|^2) \nonumber \\
&&\hspace{2cm}\leq \eps_1\|\SX\cdot b\|^2
+\frac{C}{\eps_1}(\|\{\SI-\SP\}\SX u\|^2+\|
\{\SI-\SP\}u\|^2),\label{lem.s4.fd.p3}
\end{eqnarray}
where due to Lemma \ref{lem.s4.con} and Poincar\'{e} inequality in
Proposition \ref{prop.pn}, $\|c\|^2$ was included in the dissipation
rate.

\medskip

\noindent{\it Step 2.} Estimate on $b$: \eqref{s4.f.fs}$_4$ gives
\begin{eqnarray}
  \sum_{ij=1}^d\|\SX_i b_j+\SX_j b_i\|^2 &=& \sum_{ij=1}^d (\SX_i b_j+\SX_j b_i,
  -\pa_t [\Ga_{ij}(\{\SI-\SP\}u)+2c\de_{ij}]+\Ga_{ij}(r)) \nonumber \\
  &=&-\frac{d}{dt}\sum_{ij=1}^d (\SX_i b_j+\SX_j b_i,
  \Ga_{ij}(\{\SI-\SP\}u)+2c\de_{ij}) \nonumber \\
  &&+\sum_{ij=1}^d (\SX_i \pa_t b_j+\SX_j \pa_t b_i,
  \Ga_{ij}(\{\SI-\SP\}u)+2c\de_{ij}) \nonumber \\
&&+\sum_{ij=1}^d (\SX_ib_j+\SX_j  b_i,
  \Ga_{ij}(r)).\label{lem.s4.fd.p4}
\end{eqnarray}
From \eqref{s4.f.fs}$_2$, $\pa_t b$ is replaced to obtain
\begin{eqnarray*}
&&\sum_{ij=1}^d (\SX_i \pa_t b_j+\SX_j \pa_t b_i,
  \Ga_{ij}(\{\SI-\SP\}u)+2c\de_{ij})\\
&&=2\sum_{ij=1}^d (\pa_t b_i,
  \SX_j^\ast\Ga_{ij}(\{\SI-\SP\}u)+2\SX_j^\ast c\de_{ij})\\
&&=2\sum_{ij=1}^d (-\SX_i (a+2c)+\sum_{\ell=1}^d\SX_\ell^\ast
\Ga_{\ell i}(\{\SI-\SP\}u),
  \SX_j^\ast\Ga_{ij}(\{\SI-\SP\}u)+2\SX_j^\ast c\de_{ij})\\
&&\leq \eps_2\|\SX a\|^2 + \frac{C}{\eps_2}(\|\SX^\ast c\|^2+\|\SX
c\|^2)+\frac{C}{\eps_2}\|\SX^\ast \{\SI-\SP\}u\|^2
\end{eqnarray*}
for an arbitrary constant $0<\eps_2\leq 1$ to be chosen later. The
final term on the r.h.s. of \eqref{lem.s4.fd.p4} is bounded by
\begin{equation*}
\sum_{ij=1}^d (\SX_ib_j+\SX_j  b_i,
  \Ga_{ij}(r))\leq \frac{1}{2}\sum_{ij=1}^d\|\SX_i b_j+\SX_j
  b_i\|^2+C\|\Ga(r)\|^2.
\end{equation*}
Similar to Step 1, it follows from  \eqref{lem.s4.fd.p4}  that
\begin{eqnarray}
&&\frac{d}{dt}\sum_{ij=1}^d (\SX_i b_j+\SX_j b_i,
  \Ga_{ij}(\{\SI-\SP\}u)+2c\de_{ij})+\la \sum_{i=1}^d (\|\SX
  b_i\|^2+\|b_i\|^2) \nonumber \\
&&\ \ \ \ \ \leq  \eps_2\|\SX a\|^2 + \frac{C}{\eps_2}(\|\SX
c\|^2+\|c\|^2)
 +\frac{C}{\eps_2}(\| \{\SI-\SP\}\SX u\|^2+\| \{\SI-\SP\}u\|^2),\label{lem.s4.fd.p5}
\end{eqnarray}
where due to Lemma \ref{lem.s4.con}, Korn inequalities
\eqref{thm.kn.1} and \eqref{cor.kn.2} were used.

\medskip

\noindent{\it Step 3.} Estimate on $a$: \eqref{s4.f.fs}$_2$ implies
\begin{eqnarray*}
  \|\SX a\|^2 &=& \sum_{i=1}^d (\SX_i a, -\pa_t b_i-2\SX_i c +\sum_{j=1}^d \SX_j^\ast \Ga_{ij}(
  \{\SI-\SP\}u))\\
  &=&-\frac{d}{dt} \sum_{i=1}^d (\SX_i a,  b_i)+
  \sum_{i=1}^d (\SX_i \pa_t a,  b_i)\\
  &&+\sum_{i=1}^d (\SX_i a, -2\SX_i c +\sum_{j=1}^d \SX_j^\ast \Ga_{ij}(
  \{\SI-\SP\}u)),
\end{eqnarray*}
where it further holds that
\begin{equation*}
 \sum_{i=1}^d (\SX_i \pa_t a,  b_i)= (\pa_t a,  \SX^\ast\cdot
 b)=\|\SX^\ast\cdot b\|^2\leq C \sum_{i=1}^d (\|\SX
 b_i\|^2+\|b_i\|^2)
\end{equation*}
by \eqref{s4.f.fs}$_1$, and
\begin{eqnarray*}
\sum_{i=1}^d (\SX_i a, -2\SX_i c &+&\sum_{j=1}^d \SX_j^\ast
\Ga_{ij}(
  \{\SI-\SP\}u)) \\
  &\leq & \frac{1}{2}\|\SX a\|^2 +C\|\SX c\|^2
  +C\|\SX^\ast  \{\SI-\SP\}u\|^2.
\end{eqnarray*}
Then, it follows
\begin{eqnarray}
&\dis \frac{d}{dt} \sum_{i=1}^d (\SX_i a,  b_i)+\la (\|\SX
a\|^2+\|a\|^2) \nonumber \\
&\dis \leq C \sum_{i=1}^d (\|\SX
 b_i\|^2+\|b_i\|^2)+C\|\SX c\|^2 +C (\| \{\SI-\SP\}\SX u\|^2+\|
 \{\SI-\SP\}u\|^2),\label{lem.s4.fd.p6}
\end{eqnarray}
where again due to Lemma \ref{lem.s4.con}, Poincar\'{e} inequality
in Proposition \ref{prop.pn} was used for $a$.

\medskip

Finally, define $\CE_{\rm int}(u(t))$ as in \eqref{lem.s4.fd.1}.
Then, for properly chosen constants $0<\kappa_2\ll \kappa_1\ll 1$
and small constants $\eps_1>0,\eps_2>0$, \eqref{lem.s4.fd.2} follows
from the linear combination of three inequalities
\eqref{lem.s4.fd.p3}, \eqref{lem.s4.fd.p5} and \eqref{lem.s4.fd.p6}.
This completes the proof of Lemma \ref{lem.s4.fd}.
\end{proof}

Now, for $t\geq 0$, define
\begin{equation*}
    \CE(u(t))=\|u\|^2 +\kappa_3 (\|\SX u\|^2+\|\SY
    u\|^2)+\kappa_4 \CE_{\rm int}(u(t))
\end{equation*}
where  $\CE_{\rm int}(u(t))$ is defined by \eqref{lem.s4.fd.1} and
constants $0<\kappa_4\ll \kappa_3\ll 1$ are to be chosen later. By
letting $0<\kappa_4\ll \kappa_3\ll 1$ be properly chosen, one has
\begin{equation*}
\CE(u(t))\sim \|u(t)\|_{\CH^1}^2.
\end{equation*}
On the other hand, by further properly choosing $0<\kappa_4\ll
\kappa_3\ll 1$, the linear combination of \eqref{s4.f3.p1},
\eqref{s4.f3.p2} and \eqref{lem.s4.fd.2} yields
\begin{equation*}
    \frac{d}{dt}\CE(u(t))+\la  \CE(u(t)) \leq 0,
\end{equation*}
where $\eps>0$ in \eqref{s4.f3.p2} was also chosen small enough.
Hence, \eqref{thm.f.1} follows from Gronwall inequality. This
completes the proof of Theorem \ref{thm.f} for Model 3.

\section{Appendix: Poincar\'{e} and Korn-type inequalities}

Let $V=V(x)$ be a smooth confining function over $\R^d$ with
\begin{equation}\label{iden1}
    \int_{\R^d}e^{-V(x)}dx=1.
\end{equation}
We need assumptions on $V$:
\begin{equation}\label{v-cond0}
    \frac{1}{4}|\na V|^2-\frac{1}{2}\De V \to \infty \ \
    \text{as}\ |x|\to \infty
\end{equation}
or
\begin{equation}\label{v-cond1}
    \frac{1}{4}|\na V|^2-\frac{1}{2}\De V -\sum_{ij=1}^d|\pa_i\pa_j V|\to \infty \ \
    \text{as}\ |x|\to \infty.
\end{equation}
Define operators $\SX_i$ $(1\leq i\leq d)$ associated with $V$ as in
\eqref{def.op.xy}. Throughout this section, differential operators
$\na,\De,\na^2$ denotes $\na_x,\De_x,\na^2_x$ for simplicity since
all functions are those of spatial variable. The following
Poincar\'{e} inequality in $\R^d$ has been used in this paper.

\begin{proposition}\label{prop.pn}
Let $d\geq 1$. Suppose \eqref{iden1} and \eqref{v-cond0} hold. There
is a constant $\la>0$ such that
\begin{equation}\label{prop.pn.1}
    \sum_{i=1}^d\int_{\R^d} |\SX_i a|^2 dx \geq \la \int_{\R^d}
    |a|^2 dx
\end{equation}
holds for any $a:\R^d\to \R\in L^2(\R^d)$ satisfying
\begin{equation*}
    \int_{\R^d} e^{-\frac{V(x)}{2}} a dx=0.
\end{equation*}
\end{proposition}

For the proof of the above proposition, refer to \cite{Vi} on the
basis of the constructive method, where $\la$ can be explicitly
computed.

The following Korn-type inequality in $\R^d$ was also essentially
used in the previous proof to obtain the dissipation of the
macroscopic momentum.

\begin{theorem}\label{thm.kn}
Let $d\geq 2$. Suppose  \eqref{iden1} and \eqref{v-cond1} hold.
There is a constant $\la>0$ such that
\begin{equation}\label{thm.kn.1}
    \sum_{ij=1}^{d}\int_{\R^d}|\SX_i b_j+\SX_j b_i|^2
    dx\geq \la
    \sum_{i=1}^{d} \int_{\R^d}|b_i|^2 dx
\end{equation}
holds for any $b=(b_1,b_2,\cdots,b_d):\R^d\to \R^d\in (L^2(\R^d))^d$
satisfying
\begin{equation*}
\int_{\R^d} e^{-\frac{V(x)}{2}} b(x) dx=0
\end{equation*}
and
\begin{equation*}
 \int_{\R^d} e^{-\frac{V(x)}{2}}(x_ib_j-x_jb_i)dx=0,\ \ 1\leq i
 \neq j \leq d.
\end{equation*}

\end{theorem}

\begin{proof}
We use the  contradiction argument.

\medskip

\noindent{\it Step 1.} Otherwise, for any $n\geq 1$, there is a
non-zero element $b^n=(b^n_1,b^n_2,\cdots,b^n_d)\in (L^2(\R^d))^d$
with
\begin{equation*}
\int_{\R^d} e^{-\frac{V(x)}{2}} b^n(x) dx=0
\end{equation*}
and
\begin{equation*}
 \int_{\R^d} e^{-\frac{V(x)}{2}} x\times b^n(x) dx=0,
\end{equation*}
such that
\begin{equation*}
    \sum_{ij=1}^{d}\int_{\R^d}|\SX_i b_j^n+\SX_j b_i^n|^2
    dx\leq \frac{1}{n}
    \sum_{i=1}^{d} \int_{\R^d}|b_i^n|^2 dx.
\end{equation*}
W.l.g., one can suppose
\begin{equation*}
 \sum_{i=1}^{d} \int_{\R^d}|b_i^n|^2 dx=1.
\end{equation*}
Thus,
\begin{equation*}
    \sum_{ij=1}^{d}\int_{\R^d}|\SX_i b_j^n+\SX_j b_i^n|^2
    dx\leq \frac{1}{n}.
\end{equation*}

\medskip

\noindent{\it Step 2.} We claim that there is a constant $C$
independent of $n\geq 1$ and a continuous function $w(x)$ only
depending on $V$ with
\begin{equation*}
    \inf_{x\in \R^d} w(x)\geq 1\ \text{and}\ \ w(x)\to \infty \ \
    \text{as}\ |x|\to \infty
\end{equation*}
such that
\begin{equation*}
   \sup_{n\geq 1} \int_{\R^d}|\na b^n|^2+|b^n|^2w(x)dx \leq C.
\end{equation*}
In fact, one can compute
\begin{eqnarray}
   \sum_{ij=1}^{d}\int_{\R^d}|\SX_i b_j^n+\SX_j b_i^n|^2dx &=& 2\sum_{ij=1}^{d}\int_{\R^d }|\SX_i
   b_j^n|^2dx+2\sum_{ij=1}^{d}(\SX_i b_j^n,\SX_j b_i^n) \nonumber \\
   &=&: I_1+I_2.\label{thm.kn.p0}
\end{eqnarray}
For $I_1$, it holds that
\begin{eqnarray*}
  I_1 &=&2\sum_{ij=1}^{d}\int_{\R^d }|\SX_i
   b_j^n|^2dx\\
   &=&2\sum_{ij=1}^{d}\int_{\R^d }|\pa_i b^n_j|^2+|b_j^n|^2(\frac{1}{4}|\pa_i V|^2-\frac{1}{2}\pa_i\pa_i
   V)dx\\
   &=&2\int_{\R^d } |\na b^n|^2+|b^n|^2(\frac{1}{4}|\na V|^2-\frac{1}{2}\De
   V)dx.
\end{eqnarray*}
For $I_2$, from integration by parts, it holds that
\begin{eqnarray*}
 I_2&=& 2\sum_{ij=1}^{d}(\SX_i b_j^n,\SX_j b_i^n)
 = 2\sum_{ij=1}^{d}(\frac{1}{2}\pa_i V b^n_j+\pa_i b^n_j,\frac{1}{2}\pa_j V b^n_i+\pa_j
 b^n_i)\\
 &=&2\sum_{ij=1}^{d}\left[(\frac{1}{2}\pa_i Vb^n_j,\frac{1}{2}\pa_j V b^n_i)+
 (\pa_i b^n_j,\pa_j
 b^n_i)+(\frac{1}{2}\pa_i V b^n_j,\pa_j
 b^n_i)+(\pa_i b^n_j,\frac{1}{2}\pa_j V b^n_i)\right]\\
&=&2\sum_{ij=1}^{d}\left[(\frac{1}{2}\pa_i Vb^n_i,\frac{1}{2}\pa_j V
b^n_j)+
 (\pa_i b^n_i,\pa_j
 b^n_j)\right]\\
 &&+2\sum_{ij=1}^{d}\left[-(\frac{1}{2}\pa_j\pa_i V b^n_j,
 b^n_i)-(\frac{1}{2}\pa_i V \pa_jb^n_j,
 b^n_i)-(\frac{1}{2}\pa_i \pa_j V b^n_i,b^n_j)-(\frac{1}{2} \pa_j V\pa_i
 b^n_i,b^n_j)\right]
\end{eqnarray*}
which further can be written as
\begin{eqnarray*}
 I_2
 &=&2\left[(\frac{1}{2}\na V \cdot b^n,\frac{1}{2}\na V \cdot b^n)+(\na\cdot b^n,\na\cdot b^n)
 -(\na V \cdot b^n,\na\cdot b^n)\right]\\
 &&-2(b^n,\na^2 V b^n)\\
 &=&2\|\frac{1}{2}\na V \cdot b^n-\na\cdot b^n\|^2-2(b^n,\na^2 V
 b^n).
\end{eqnarray*}
Thus, it follows
\begin{eqnarray*}
  I_1+I_2 &=&2\int_{\R^d } |\na b^n|^2+|b^n|^2(\frac{1}{4}|\na V|^2-\frac{1}{2}\De
   V)dx\\
   &&+\|\frac{1}{2}\na V \cdot b^n-\na\cdot b^n\|^2-2(b^n,\na^2 V
 b^n).
\end{eqnarray*}
Then, this implies
\begin{eqnarray}
  I_1+I_2
 &\geq & 2\int_{\R^d } |\na b^n|^2+|b^n|^2(\frac{1}{4}|\na V|^2-\frac{1}{2}\De
   V-\sum_{ij=1}^d |\pa_i\pa_j V|)dx \nonumber  \\
   &&+\|\frac{1}{2}\na V \cdot b^n-\na\cdot b^n\|^2.\label{thm.kn.p1}
\end{eqnarray}
Now, define
\begin{equation*}
    w(x)=\max\{\frac{1}{4}|\na V|^2-\frac{1}{2}\De
   V-\sum_{ij=1}^d |\pa_i\pa_j V|,1\}.
\end{equation*}
Due to \eqref{v-cond1}, $w(x)\to\infty$ as $|x|\to\infty$. Choosing
$R>0$ large enough, one has
\begin{equation*}
    w(x)=\frac{1}{4}|\na V|^2-\frac{1}{2}\De
   V-\sum_{ij=1}^d |\pa_i\pa_j V|
\end{equation*}
for $|x|\geq R$. Then
\begin{eqnarray*}
&&\int_{\R^d }|b^n|^2(\frac{1}{4}|\na V|^2-\frac{1}{2}\De
   V-\sum_{ij=1}^d |\pa_i\pa_j V|)dx=\int_{|x|\geq R} +\int_{|x|\leq R}\\
   &&=\int_{\R^d }|b^n|^2w(x)dx-\int_{|x|\leq R}|b^n|^2w(x)dx+\int_{|x|\leq
   R}\\
   & &\geq\int_{\R^d }|b^n|^2w(x)dx -C_{R,V},
\end{eqnarray*}
where
\begin{equation*}
C_{R,V}=\sup_{|x|\leq R}w(x)-\inf_{|x|\leq R}(\frac{1}{4}|\na
V|^2-\frac{1}{2}\De
   V-\sum_{ij=1}^d |\pa_i\pa_j V|).
\end{equation*}
Then, it follows that
\begin{eqnarray*}
 &\dis\sum_{ij=1}^{d}\int_{\R^d }|\SX_i b_j^n+\SX_j b_i^n|^2dx  = I_1+I_2\\
 &\dis \geq 2\int_{\R^d }|\na b^n|^2+|b^n|^2 w(x)dx+\|\frac{1}{2}\na V \cdot b^n-\na\cdot
 b^n\|^2-2C_{R,V}.
\end{eqnarray*}
That is,
\begin{eqnarray*}
 \int_{\R^d }|\na b^n|^2+|b^n|^2 w(x)dx &\leq & \frac{1}{2}\sum_{ij=1}^{d}\int_{\R^d }|\SX_i b_j^n+\SX_j b_i^n|^2dx
 +C_{R,V}\\
 &\leq & \frac{1}{2n} +C_{R,V}\leq  \frac{1}{2} +C_{R,V}.
\end{eqnarray*}
Hence, the claim follows.

\medskip

\noindent{\it Step 3.} From the Rellich-Kondrachov compactness
theorem, up to a subsequence, $(b^n)_{n \geq 1}$ strongly converges
to some function $b\in (L^2(\R^d ))^d$, i.e.
\begin{equation*}
    \|b^n-b\|\to 0 \ \text{as}\ n\to\infty.
\end{equation*}
Thus, $b$ satisfies
\begin{eqnarray}
&\dis \|b\|=1,\label{lire0}
\end{eqnarray}
and
\begin{eqnarray}
&\dis X_ib_j+X_jb_i=0,\ \ 1\leq i,j\leq d,\label{lire1}\\
&\dis \int_{\R^d } e^{-\frac{V(x)}{2}} b(x) dx=0,\label{lire2}\\
&\dis \int_{\R^d } e^{-\frac{V(x)}{2}}(x_i b_j-x_jb_i) dx=0,\ \
1\leq i\neq j \leq d.\label{lire3}
\end{eqnarray}

\medskip

\noindent{\it Step 4.} We claim that \eqref{lire1}, \eqref{lire2}
and \eqref{lire3} imply $b=0$, which is a contradiction to
\eqref{lire0}. In fact, set
\begin{equation*}
    b_i(x)=e^{-\frac{V}{2}} \tilde{b}_i(x),\ \ i=1,2,\cdots,d.
\end{equation*}
From $X_ib_i=0$,
\begin{equation*}
    \pa_i \tilde{b}_i=0.
\end{equation*}
Then, $ \tilde{b}_i(x)$ is independent of $x_i$. Thus, set
$\tilde{b}_i=\tilde{b}_i(\tilde{x}_i)$, where $\tilde{x}_i$ is a
variable $(x_1,x_2,\cdots,x_d)$ excluding $x_i$.  From \eqref{lire1}
for $i\neq j$, one has
\begin{equation}\label{lire4}
    \pa_i \tilde{b}_j(\tilde{x}_j)+\pa_j \tilde{b}_i(\tilde{x}_i)=0,
    \ \ i\neq j.
\end{equation}
This implies that $\tilde{b}_i(\tilde{x}_i)$ is linear in each $x_j$
$(j\neq i)$. Thus,
\begin{equation*}
\tilde{b}_i(\tilde{x}_i)=\sum_{\{\al:\, \al_i=0,\al_j=0\,
\text{or}\, 1, j\neq i\}} C_{\al}^i x^\al.
\end{equation*}
It suffices to prove all coefficients $ C_{\al}^i$ vanish. In what
follows, fix $i\in\{1,\cdots,d\}$.

\medskip

\noindent{\it Case 1.} $|\al|=0$. Due to \eqref{lire2},
\begin{equation*}
    C_\al^i=0.
\end{equation*}

\medskip

\noindent{\it Case 2.} $2\leq |\al|\leq d-1$. Let $\ell=|\al|$ and
take
\begin{equation*}
    \{i_1,i_2,\cdots,i_{\ell}\}\subset \{1,2,\cdots,d\}.
\end{equation*}
Consider the coefficient of the monomial
\begin{equation*}
    \Pi_{j=1}^kx_{i_j}.
\end{equation*}
It only appears in $\tilde{b}_j$ with $j\in
\{1,2,\cdots,d\}\backslash \{i_1,i_2,\cdots,i_{\ell}\}$. Take
\begin{equation*}
    i_{\ell+1}\in \{1,2,\cdots,d\}\backslash
    \{i_1,i_2,\cdots,i_{\ell}\}.
\end{equation*}
Consider the following three functions
\begin{eqnarray*}
  \tilde{b}_{i_{\ell+1}} &=& \cdots+ C_1 x_{i_1}x_{i_2}\cdots
  x_{i_{\ell-1}}x_{i_{\ell}}+\cdots,\\
 \tilde{b}_{i_{\ell}} &=& \cdots+ C_2 x_{i_1}x_{i_2}\cdots
  x_{i_{\ell-1}}x_{i_{\ell+1}}+\cdots,\\
   \tilde{b}_{i_{\ell-1}} &=& \cdots+ C_3 x_{i_1}x_{i_2}\cdots
  x_{i_{\ell-2}}x_{i_{\ell}}x_{i_{\ell+1}}+\cdots,
\end{eqnarray*}
where $C_1,C_2,C_3$ are the corresponding constant coefficients. Due
to \eqref{lire4},
\begin{equation*}
    C_1+C_2=0,\ \ C_2+C_3=0,\ \ C_3+C_1=0.
\end{equation*}
Thus, $C_1=C_2=C_3=0$. The coefficient of $ \Pi_{j=1}^kx_{i_j}$ is
zero.  This proves
\begin{equation*}
    C_\al^i=0
\end{equation*}
when $2\leq |\al|\leq d-1$.

\medskip

\noindent{\it Case 3.} $|\al|=1$. Take $j\neq i$. Consider the
function
\begin{equation*}
    \tilde{b}_j=\cdots+C_i x_i+\cdots,
\end{equation*}
where $C_i$ is a constant. Due to  \eqref{lire4}, $\tilde{b}_i$ is
in the form of
\begin{equation*}
\tilde{b}_i=\cdots-C_i x_j+\cdots.
\end{equation*}
By using \eqref{lire3}, one has
\begin{equation*}
    0=\int_{\R^d} e^{-V}(x_i\tilde{b}_j-x_j \tilde{b}_i)dx=C_i\int_{\R^d}
    e^{-V}(x_i^2+x_j^2)dx.
\end{equation*}
Then, $C_i=0$. Thus,
\begin{equation*}
    C_{\al}^i=0
\end{equation*}
when $|\al|=1$. This also completes the proof of Theorem
\ref{thm.kn}.
\end{proof}

\begin{corollary}\label{cor.kn}
Let $d\geq 2$. Suppose  \eqref{iden1} and \eqref{v-cond1} hold.
Further assume that there is a constant $C$ such that
\begin{equation}\label{cor.kn.1}
    |\na^2 V|^2\leq C(|\na V|^2+1)
\end{equation}
for all $x\in \R^d$. Then, there is a constant $\la>0$ such that
\begin{equation}\label{cor.kn.2}
    \sum_{ij=1}^{d}\int_{\R^d}|\SX_i b_j+\SX_j b_i|^2
    dx\geq \la
    \sum_{ij=1}^{d} \int_{\R^d}|\SX_ib_j|^2 dx
\end{equation}
holds for any $b=(b_1,b_2,\cdots,b_d):\R^d\to \R^d\in (L^2(\R^d))^d$
satisfying
\begin{equation*}
\int_{\R^d} e^{-\frac{V(x)}{2}} b(x) dx=0
\end{equation*}
and
\begin{equation*}
 \int_{\R^d} e^{-\frac{V(x)}{2}}(x_ib_j-x_jb_i)dx=0,\ \ 1\leq i
 \neq j \leq d.
\end{equation*}

\end{corollary}

\begin{proof}
Recall from \eqref{thm.kn.p0} and \eqref{thm.kn.p1} that one has the
inequality
\begin{eqnarray}
 &\dis \sum_{ij=1}^{d}\int_{\R^d}|\SX_i b_j+\SX_j b_i|^2dx \nonumber \\
 &\dis \geq
 2\int_{\R^d}|\na b|^2 +|b|^2 (\frac{1}{4}|\na V|^2-\frac{1}{2}\De V-\sum_{ij=1}^d|\pa_i\pa_j
 V|) dx.\label{cor.kn.p1}
\end{eqnarray}
From the assumption \eqref{cor.kn.1},
\begin{eqnarray*}
  |\na V|^2 &=& 4 (\frac{1}{4}|\na V|^2-\frac{1}{2}\De V)+2 \De V\\
  &\geq & 4 (\frac{1}{4}|\na V|^2-\frac{1}{2}\De V) -\eps |\na^2
  V|^2-C_\eps\\
&\geq & 4 (\frac{1}{4}|\na V|^2-\frac{1}{2}\De V) -C\eps |\na
  V|^2-C_\eps
\end{eqnarray*}
for any $\eps>0$. This implies
\begin{equation}\label{cor.kn.p2}
 |\na V|^2\geq \la (\frac{1}{4}|\na V|^2-\frac{1}{2}\De V)-C_\la
\end{equation}
for some constant $0< \la<4$. Similarly, it holds that
\begin{eqnarray}
\frac{1}{4}|\na V|^2-\frac{1}{2}\De V-\sum_{ij=1}^d|\pa_i\pa_j
 V| &\geq &\frac{1}{4}|\na V|^2-\eps |\na^2 V|^2-C_\eps \nonumber \\
 &\geq &(\frac{1}{4}-C\eps)|\na V|^2-C_\eps\label{cor.kn.p3}
\end{eqnarray}
for any $\eps>0$. Then, combining \eqref{cor.kn.p2} and
\eqref{cor.kn.p3},  it follows
\begin{equation*}
\frac{1}{4}|\na V|^2-\frac{1}{2}\De V-\sum_{ij=1}^d|\pa_i\pa_j
 V| \geq \la (\frac{1}{4}|\na V|^2-\frac{1}{2}\De V)-C_\la
\end{equation*}
for some constant $0<\la<1$. Plugging the above inequality into
\eqref{cor.kn.p1} gives
\begin{eqnarray*}
 &\dis \sum_{ij=1}^{d}\int_{\R^d}|\SX_i b_j+\SX_j b_i|^2dx \\
 &\dis \geq
 \la \int_{\R^d}|\na b|^2 +|b|^2 (\frac{1}{4}|\na V|^2-\frac{1}{2}\De V) dx-C_\la \int_{\R^d}|b|^2dx.
\end{eqnarray*}
From the Korn-type inequality \eqref{thm.kn.1}, it further follows
\begin{eqnarray*}
 \sum_{ij=1}^{d}\int_{\R^d}|\SX_i b_j+\SX_j b_i|^2dx
 &\geq&
 \la \int_{\R^d}|\na b|^2 +|b|^2 (\frac{1}{4}|\na V|^2-\frac{1}{2}\De V)
 dx\\
 &=&\la  \sum_{ij=1}^{d} \int_{\R^d}|\SX_ib_j|^2 dx.
\end{eqnarray*}
This proves \eqref{cor.kn.2} and thus completes the proof of
Corollary \ref{cor.kn}.
\end{proof}

We conclude this paper with a remark that Proposition \ref{prop.pn},
Theorem \ref{thm.kn} and Corollary \ref{cor.kn} can be described in
terms of the probability measure $d\mu=e^{-V(x)}dx$ as in \cite{Vi},
where $V$ satisfies the corresponding conditions. In fact,
equivalently, \eqref{prop.pn.1} become
\begin{equation*}
    \int_{\R^d}|\na a|^2d\mu \geq \la \int_{\R^d}|a|^2 d\mu
\end{equation*}
for $a$ with $\int_{\R^d} a d\mu =0$. Similarly, \eqref{thm.kn.1}
and \eqref{cor.kn.2} are respectively equivalent with
\begin{equation*}
    \sum_{ij=1}^d\int_{\R^d} |\pa_ib_j+\pa_j b_i|^2d\mu\geq \la
    \sum_{i=1}^d  \int_{\R^d} |b_i|^2d\mu
\end{equation*}
and
\begin{equation*}
    \sum_{ij=1}^d\int_{\R^d} |\pa_ib_j+\pa_j b_i|^2d\mu\geq \la
    \sum_{ij=1}^d  \int_{\R^d} |\pa_ib_j|^2d\mu
\end{equation*}
for $b=(b_1,b_2,\cdots,b_d)$ with
\begin{equation*}
    \int_{\R^d} b \,d\mu=0 \ \ \text{and}\ \  \int_{\R^d} x\times
    b\,
    d\mu=0.
\end{equation*}

\vspace{1cm}

\noindent {\bf Acknowledgments:}\,\, The author thanks  Jean
Dolbeault and Cl\'{e}ment Mouhot for their helpful discussions and
suggestions for the development of this work. The author would like
to thank Peter Markowich and Massimo Fornasier for their strong
support during the study of this work at RICAM. Warm thanks go to
Seiji Ukai, Tong Yang, Huijiang Zhao and Changjiang Zhu for their
continuous encouragement.

%\newpage

\end{document}